\documentclass[11pt]{amsart}
\usepackage{appendix}
\usepackage{latexsym,amssymb}
\usepackage{amsthm}
\usepackage{mathrsfs}
\usepackage{amsbsy}
\usepackage[latin1]{inputenc}
\usepackage[stable, hang, flushmargin]{footmisc}
\usepackage{hyperref}
\usepackage{mathtools}
\usepackage{graphicx}
\usepackage{fancyhdr}
\usepackage[numbers,sort&compress]{natbib}

\theoremstyle{plain}
\newtheorem{theorem}{Theorem}[subsection]

\newtheorem{corollary}[theorem]{Corollary}

\newtheorem{lemma}[theorem]{Lemma}

\newtheorem{question}{Question}[section]
\newtheorem{proposition}[theorem]{Proposition}

\theoremstyle{definition}

\newtheorem{remark}[theorem]{Remark}
\newtheorem{definition}[theorem]{Definition}
\newtheorem*{definition-no}{Definition}

\newtheorem{problem}[question]{Problem}
\newtheorem*{claim-no}{Claim}
\newtheorem{innercustomthm}{Theorem}
\newenvironment{customthm}[1]
  {\renewcommand\theinnercustomthm{#1}\innercustomthm}
  {\endinnercustomthm}
\DeclareMathOperator*{\foo}{\raisebox{-0.6ex}{\scalebox{2.5}{$\ast$}}}

\fancypagestyle{firststyle}
{
   \fancyhead{}
   \fancyfoot{}}

\begin{document}
\title[Polish groupoids and functorial complexity]{Polish groupoids and
functorial complexity}
\author{Martino Lupini}
\address{ Department of Mathematics and Statistics\\
N520 Ross, 4700 Keele Street\\
Toronto Ontario M3J 1P3, Canada, and Fields Institute for Research in
Mathematical Sciences\\
222 College Street\\
Toronto ON M5T 3J1, Canada.}
\email{mlupini@yorku.ca}
\urladdr{http://www.lupini.org/}
\thanks{The author was supported by the York University Elia Scholars
Program. This work was completed when the author was attending the Thematic
Program on Abstract Harmonic Analysis, Banach and Operator Algebras at the
Fields Institute. The hospitality of the Fields Institute is gratefully
acknowledged.}
\dedicatory{}

\begin{abstract}
We introduce and study the notion of functorial Borel complexity for Polish
groupoids. Such a notion aims at measuring the complexity of classifying the objects
of a category in a constructive and functorial way. In the particular case
of principal groupoids such a notion coincide with the usual Borel
complexity of equivalence relations. Our main result is that on one hand for
Polish groupoids with essentially treeable orbit equivalence relation,
functorial Borel complexity coincides with the Borel complexity of the
associated orbit equivalence relation. On the other hand for every countable
equivalence relation $E$ that is not treeable there are Polish groupoids
with different functorial Borel complexity both having $E$ as orbit
equivalence relation. In order to obtain such a conclusion we generalize
some fundamental results about the descriptive set theory of Polish group
actions to actions of Polish groupoids, answering a question of Arlan
Ramsay. These include the Becker-Kechris results on Polishability of Borel $%
G $-spaces, existence of universal Borel $G$-spaces, and characterization of
Borel $G$-spaces with Borel orbit equivalence relations.
\end{abstract}

\maketitle
\tableofcontents

\addtocontents{toc}{\protect\enlargethispage{\baselineskip}} %
\thispagestyle{firststyle}


\section{Introduction\label{Section: introduction}}

Classification of mathematical structures is one of the main components of
modern mathematics. It is safe to say that most results in mathematics can
be described as providing an explicit classification of a class of
mathematical objects by a certain type of invariants.

In the last 25 years the notion of constructive classification has been
given a rigorous formulation in the framework of invariant complexity
theory. In this context a classification problem is regarded as an
equivalence relation on a standard Borel space (virtually all classification
problems in mathematics fit into this category). The concept of constructive
classification is formalized by the notion of \emph{Borel reduction}. A
Borel reduction from an equivalence relation $E$ on $X$ to an equivalence
relation $E^{\prime }$ on $X^{\prime }$ is a \emph{Borel} function $%
f:X\rightarrow X^{\prime }$ with the property that, for every $x,y\in X$,%
\begin{equation*}
xEy\text{ if and only if }f(x)E^{\prime }f(y)\text{.}
\end{equation*}%
In other words $f$ is a Borel assignment of complete invariants for $E$ that
are equivalence classes of $E^{\prime }$. The existence of such a function
can be interpreted as saying that classifying the objects of $X^{\prime }$
up to $E^{\prime }$ is at least as complicated as classifying the objects of 
$X$ up to $E$. This offers a notion of comparison between the complexity of
different classification problems.

Several natural equivalence relation can then be used as benchmarks to
measure the complexity of classification problems. Perhaps the most obvious
such a benchmark is the relation of equality $=_{\mathbb{R}}$ of real
numbers. This gives origin to the basic dichotomy smooth vs.\ non-smooth: an
equivalence relation is \emph{smooth }if it is Borel reducible to $=_{%
\mathbb{R}}$. (The real numbers can here be replaced by any other
uncountable standard Borel space.) Beyond smoothness the next fundamental
benchmark is classifiability by countable structures. Here the test is Borel
reducibility to the relation of isomorphism within some class of countable
first order structures, such as (ordered) groups, rings, etc. Equivalently
one can consider orbit equivalence relation associated with Borel actions of
the Polish group $S_{\infty }$ of permutations of $\mathbb{N}$. Replacing $%
S_{\infty }$ with an arbitrary Polish group yields the notion of equivalence
relation classifiable by orbits of a Polish group action.

This framework allows one to build a hierarchy between different
classification problems. Many efforts have been dedicated to the attempt to
draw a picture as complete as possible of classification problems in
mathematics and their relative complexity. To this purpose powerful tools
such as Hjorth's theory of turbulence~\cite{hjorth_classification_2000} have
been developed in order to disprove the existence of Borel reduction between
given equivalence relations, and to distinguish between the complexity of
different classification problems. This can be interpreted as a way to
formally exclude the possibility of a full classification of a certain class
of objects by means of a given type of invariants. For example the relation
of isomorphism of simple separable C*-algebras has been shown to transcend
countable structures in~\cite{farah_turbulence_2014}; see also~\cite%
{sabok_completeness_2013}. Similar results have been obtained for several
other equivalence relations, such as affine homeomorphism of Choquet
simplexes~\cite{farah_turbulence_2014}, conjugacy of unitary operators on
the infinite dimensional separable Hilbert space~\cite{kechris_strong_2001},
conjugacy of ergodic measure-preserving transformations of the Lebesgue
space~\cite{foreman_anti-classification_2004}, conjugacy of homeomorphisms
of the unit square~\cite{hjorth_classification_2000}, conjugacy of
irreducible representations of non type I groups~\cite%
{hjorth_non-smooth_1997} or C*-algebras~\cite{farah_dichotomy_2012,
kerr_turbulence_2010}, conjugacy and unitary equivalence of automorphisms of
classifiable simple separable C*-algebras \cite{kerr_Borel_2014,
lupini_unitary_2014}, isometry of separable Banach spaces~\cite%
{melleray_computing_2007} and complete order isomorphism of separable
operator systems. Furthermore the relations of isomorphism and Lipschitz
isomorphisms of separable Banach spaces, topological isomorphism of
(abelian) Polish groups, uniform homeomorphism of complete separable metric
spaces~\cite{ferenczi_complexity_2009}, and the relation of completely
bounded isomorphism of separable operator spaces \cite%
{argerami_classification_2014} have been shown to be not classifiable by the
orbits of a Polish group action (and in fact to have maximal complexity
among analytic equivalence relations). An exhaustive introduction to
invariant complexity theory can be found in~\cite{gao_invariant_2009}.

Considering how helpful the theory of Borel complexity has been so far in
giving us a clear understanding of the relative complexity of classification
problems in mathematics, it seems natural to look at refinements to the
notion of Borel reducibility, that can in some situations better capture the
notion of explicit classification from the practice of mathematics. Such a
line of research has been suggested in \cite{farah_logic_2014}, where the
results of the present paper has been announced. This is the case for
example when the classification problem under consideration concerns a
category. In this case it is natural to ask to the classifying map to be 
\emph{functorial}, and to assign invariants not only to the objects of the
category, but also to the \emph{morphisms}. This is precisely what happens
in many explicit examples of classification results in mathematics. In fact
in many such examples the consideration of invariants of morphisms is
essential to the proof. This is particularly the case in the Elliott
classification program of simple C*-algebras, starting from Elliott's
seminal paper of AF algebras~\cite{elliott_classification_1976}. Motivated
by similar considerations, Elliott has suggested in~\cite%
{elliott_towards_2010} an abstract approach to classification by functors.
In this paper we bring Elliott's theory of functorial classification within
the framework of Borel complexity theory. For simplicity we consider only
categories where every arrow is invertible, called groupoids. Such
categories will be assumed to have a global Borel structure that is at least
analytic, and makes the set of objects (identified with their identity
arrows) a standard Borel space. In the particular case when between any two
objects there is at most one arrow (principal groupoids) these are precisely
the analytic equivalence relations. One can then consider the natural
constructibility requirement for classifying functors, which is being Borel
with respect to the given Borel structures. This gives rise to the notion of
functorial Borel complexity, which in the particular case of principal
groupoids is the usual notion of Borel complexity.

In this article we study such a notion of functorial Borel complexity for
groupoids, focusing on the case of Polish groupoids. These are the groupoids
where the Borel structure is induced by a topology that makes composition
and inversion of arrows continuous and open, and has a basis of open sets
which are Polish in the relative topology. These include all Polish groups,
groupoids associated with Polish group actions, and locally compact
groupoids~\cite[Definition 2.2.2]{paterson_groupoids_1999}. The latter ones
include the holonomy groupoids of foliations and the tangent groupoids of
manifolds~\cite[Chapter 2]{paterson_groupoids_1999}, the groupoids of
row-finite directed graphs~\cite{kumjian_graphs_1997}, the localization
groupoids of actions of countable inverse semigroups~\cite[Chapter 4]%
{paterson_groupoids_1999}. The main results of the present paper assert
that, for Polish groupoids with essentially countable equivalence relation,
the existence of a Borel reducibility between the groupoids is equivalent to
the Borel reducibility of the corresponding orbit equivalence relations. On
the other hand for every countable equivalence relation $E$ that is not
treeable there are two Polish groupoid with orbit equivalence relation $E$
that have distinct functorial Borel complexity; see Section~\ref{Section:
functorial reducibility and treeable}. This shows that Borel reducibility of
groupoids provides a finer notion of complexity than the usual Borel
reducibility of equivalence relations. Having a finer notion of complexity
is valuable, because it allows one to further distinguish between the
complexity of problems that, in the usual framework, turn out to have the
same complexity. An example of this phenomenon occurs in the classification
problem for C*-algebras, where it turns out~\cite{farah_turbulence_2014,
elliott_isomorphism_2013, sabok_completeness_2013} that classifying
arbitrary separable C*-algebras is as difficult as classifying the
restricted class of C*-algebras that are considered to be well behaved
(precisely the amenable simple C*-algebras, or even more restrictively the
simple C*-algebras that can be obtained as direct limits of interval
algebras).

In order to prove the above mentioned characterization of essentially
treeable equivalence relations we will generalize some fundamental results
of the theory of actions of Polish groups to actions of Polish groupoids,
answering a question of Ramsay from~\cite{ramsay_polish_2000}. These include
the Becker-Kechris results on Polishability of Borel $G$-spaces~\cite[%
Chapter 5]{becker_descriptive_1996}, existence of universal Borel $G$-spaces 
\cite[Section 2.6]{becker_descriptive_1996}, and characterization of Borel $%
G $-spaces with Borel orbit equivalence relation~\cite[Chapter 7]%
{becker_descriptive_1996}. The fundamental technique employed is a
generalization of the \emph{Vaught transform}~\cite{vaught_invariant_1974}
from actions of Polish groups to actions of Polish groupoids.

This paper is organized as follows: In Section~\ref{Section: notation} we
recall some background notions, introduce the notation to be used in the
rest, and state the basic properties of the Vaught transform for actions of
Polish groupoids. In Section~\ref{Section: Glimm-Effros} we generalize the
local version of Effros' theorem from Polish group actions to actions of
Polish groupoids, and infer the Glimm-Effros dichotomy for Polish groupoids
and Borel reducibility, refining results from~\cite{ramsay_mackey-glimm_1990}%
. Section~\ref{Section: better} contains the proof of the Polishability
result for Borel $G$-spaces, showing that any Borel $G$-space is isomorphic
to a Polish $G$-space, where $G$ is a Polish groupoid. A characterization
for Borel $G$-spaces with Borel orbit equivalence relation is obtained as a
consequence in Section~\ref{Section: Borel orbit}. Section~\ref{Section:
universal} contains the construction of a universal Borel $G$-space for a
given Polish groupoid $G$, generalizing~\cite[Section 2.6]%
{becker_descriptive_1996}. Section~\ref{Section: countable} considers
countable Borel groupoids, i.e.\ analytic groupoids with only countably many
arrows with a given source. It is shown that every such a groupoid has a
Polish groupoid structure compatible with its Borel structure. In particular
all results about Polish groupoids apply to countable Borel groupoids.
Finally in Section~\ref{Section: functorial reducibility and treeable} the
above mentioned characterization of essentially treeable equivalence
relations in terms of Borel reducibility is proved.

This paper includes an appendix written by Anush Tserunyan. In such an
appendix it is proved that the Effros Borel structure on the space of closed
subsets of a Polish groupoid is standard. We are grateful to Anush Tserunyan
for letting us include her result here.

Looking at Polish groupoids was first suggested to us by George Elliott in
occasion of a joint work with Samuel Coskey and Ilijas Farah. We would like
to thank all of them, as well as Marcin Sabok and Anush Tserunyan, for
several helpful conversations.

\section{Notation and preliminaries}

\label{Section: notation}

\subsection{Descriptive set theory}

A \emph{Polish space} is a separable and completely metrizable topological
space. Equivalently a topological space is Polish if it is T$_{1}$, regular,
second countable, and \emph{strong Choquet}~\cite[Theorem 8.18]%
{kechris_classical_1995}. A subspace of a Polish space is Polish with
respect to the subspace topology if and only if it is a $G_{\delta }$~\cite[%
Theorem 3.11]{kechris_classical_1995}.

A \emph{standard Borel space} is a space endowed with a $\sigma $-algebra
which is the $\sigma $-algebra of Borel sets with respect to some Polish
topology. An \emph{analytic space} is a space endowed with a countably
generated $\sigma $-algebra which is the image of a standard Borel space
under a Borel function. A subset of a standard Borel space is analytic if it
is an analytic space with the relative standard Borel structure. A subset of
a standard Borel space is \emph{co-}$\emph{analytic}$ if its complement is
analytic. It is well known that for a subset of a standard Borel space it is
equivalent being Borel and being both analytic and co-analytic~\cite[Theorem
14.7]{kechris_classical_1995}. If $X,Y$ are standard Borel space and $A$ is
a subset of $X\times Y$, then for $x\in X$ the section%
\begin{equation*}
\left\{ y\in Y:(x,y)\in A\right\}
\end{equation*}%
is denoted by $A_{x}$. The \emph{projection} of $A$ onto the first
coordinate is%
\begin{equation*}
\left\{ x\in X:A_{x}\neq \varnothing \right\} \text{,}
\end{equation*}%
while the \emph{co-projection }of $A$ is%
\begin{equation*}
\left\{ x\in X:A_{x}=Y\right\} \text{.}
\end{equation*}%
The projection of an analytic set is analytic, while the co-projection of a
co-analytic set is co-analytic.

If $X$ is a Polish space, then the space of closed subsets of $X$ is denoted
by $F(X)$. The \emph{Effros Borel structure} on $F(X)$ is the $\sigma $%
-algebra generated by the sets%
\begin{equation*}
\left\{ F\in F(X):F\cap U\neq \varnothing \right\}
\end{equation*}%
for $U\subset X$ open. This makes $F(X)$ a standard Borel space~\cite[%
Section 12.C]{kechris_classical_1995}.

Recall that a subset $A$ of a Polish space $X$ has the \emph{Baire property}
if there is an open subset $U$ of $X$ such that the symmetric difference $%
A\bigtriangleup U$ is meager~\cite[Definition 8.21]{kechris_classical_1995}.
It follows from~\cite[Corollary 29.14]{kechris_classical_1995} that any
analytic subset of $X$ has the Baire property.

A topological space $X$ is a \emph{Baire space} if every nonempty open
subset of $X$ is not meager. Every completely metrizable topological space
is a Baire space; see~\cite[Theorem 8.4 ]{kechris_classical_1995}.

If $X,Y$ are standard Borel spaces, then we say that $Y$ is \emph{fibred}
over $X$ if there is a Borel surjection $p:Y\rightarrow X$. If $x\in X$,
then the inverse image of $x$ under $p$ is called the $x$-fiber of $Y$ and
denoted by $Y_{x}$. If $Y_{0},Y_{1}$ are fibred over $X$, then the fibred
product 
\begin{equation*}
Y_{0}\ast Y_{1}=\left\{ \left( y_{0},y_{1}\right)
:p_{0}(y_{0})=p_{1}(y_{1})\right\}
\end{equation*}%
is naturally fibred over $X$. Similarly if $\left( Y_{n}\right) _{n\in
\omega }$ is a sequence of Borel spaces fibred over $X$ we define 
\begin{equation*}
\foo_{n\in \omega }Y_{n}=\left\{ (y_{n})_{n\in \omega }:p(y_{n})=p(y_{m})%
\text{ for }n,m\in \omega \right\}
\end{equation*}%
which is again fibred over $X$. A \emph{Borel fibred map} from $Y_{0}$ to $%
Y_{1}$ is a Borel function $\varphi :Y_{0}\rightarrow Y_{1}$ which sends
fibers to fibers, i.e.\ $p_{1}\circ \varphi =p_{0}$.

If $E$ is an equivalence relation on a standard Borel space $X$, then a
subset $T$ of $X$ is a \emph{transversal} for $E$ if it intersects any class
of $E$ in exactly one point. A \emph{selector} for $E$ is a Borel function $%
\sigma :X\rightarrow X$ such that $\sigma (x)Ex$ for every $x\in X$ and $%
\sigma (x)=\sigma (y)$ whenever $xEy$.

\subsection{Locally Polish spaces\label{Subsection: locally Polish spaces}}

\begin{definition}
\label{Definition: locally Polish space}A \emph{locally Polish space} is a
topological space with a countable basis of open sets which are Polish
spaces in the relative topology.
\end{definition}

By~\cite[Theorem 8.18]{kechris_classical_1995} a locally Polish space is T$%
_{1}$, second countable, and strong Choquet. Moreover it is a Polish space
if and only if it is regular. It follows from~\cite[Lemma 3.11]%
{kechris_classical_1995} that a $G_{\delta }$ subspace of a locally Polish
space is locally Polish.

Suppose that $X$ is a locally Polish space. Denote by $F(X)$ the space of
closed subsets of $X$. The Effros Borel structure on $F(X)$ is the $\sigma $%
-algebra generated by the sets of the form%
\begin{equation*}
\left\{ F:F\cap U\neq \varnothing \right\}
\end{equation*}%
for $U\subset X$ open. It is shown in the Appendix that the Effros Borel
structure on $F(X)$ is standard.

One can deduce from this that the Borel $\sigma $-algebra of $X$ is
standard. In fact the function%
\begin{eqnarray*}
X &\rightarrow &F(X) \\
x &\mapsto &\left\{ x\right\}
\end{eqnarray*}%
is clearly a Borel isomorphism onto the set $F_{1}(X)$ of closed subsets of $%
F(X)$ containing exactly one element. It is therefore enough to show that $%
F_{1}(X)$ is a Borel subset of $F(X)$. Fix a countable basis $\mathcal{A}$
of open Polish subsets of $X$. Suppose also that for every $U\in \mathcal{A}$
it is fixed a compatible complete metric $d_{U}$ on $U$. Observe that $%
F_{1}(X)$ contains precisely the closed subsets $F$ of $X$ such that $F\cap
U\neq \varnothing $ for some $U\in \mathcal{A}$ and for every $U\in \mathcal{%
A}$ such that $F\cap U\neq \varnothing $ and for every $n\in \omega $ there
is $W\in \mathcal{A}$ such that $cl(W)\subset U$, $diam_{U}\left( cl\left( 
\overline{W}\right) \right) <2^{-n}$ and $F\cap \left( X\backslash
cl(W)\right) =\varnothing $, where $diam_{U}\left( cl\left( \overline{W}%
\right) \right) $ is the diameter of $\overline{W}$ with respect to the
metric $d_{U}$. This shows that $F_{1}(X)$ is a Borel subset of $X$.

\subsection{The Effros fibred space\label{Subsection: Effros fibred space}}

Suppose that $Z$ is a locally Polish space, $X$ is a Polish space, and $%
p:Z\rightarrow X$ is a continuous open surjection. For $x\in X$ denote by $%
Z_{x}$ the inverse image of $x$ under $p$. Define $F^{\ast }(Z)$ to be the
space of \emph{nonempty }subsets of $Z$ endowed with the Effros Borel
structure. Define $F^{\ast }(Z,X)$ to be the Borel subset of closed subsets
of $Z$ contained in $Z_{x}$ for some $x\in X$. The Borel function from $%
F^{\ast }(Z,X)$ onto $X$ assigning to an element $F$ of $F^{\ast }(Z,X)$ the
unique $x\in X$ such that $F\subset Z_{x}$ endows $F^{\ast }(Z,X)$ with the
structure of fibred Borel space. The obvious embedding of $F^{\ast }(Z_{x})$
into $F^{\ast }(Z,X)$ is a Borel isomorphism onto the $x$-fiber of $F^{\ast
}(Z,X)$.

Consider the set $\left\{ \varnothing _{x}:x\in X\right\} $ endowed with the
Borel structure obtained from the bijection $x\leftrightarrow \varnothing
_{x}$. Define $F(Z,X)$ to be the disjoint union of $F^{\ast }(Z,X)$ with $%
\left\{ \varnothing _{x}:x\in X\right\} $, which is again fibred over $X$ in
the obvious way. Moreover the $x$-fiber of $F(Z,X)$ is now naturally
isomorphic to the space $F(Z)$ of (possibly empty) subsets of $Z_{x}$. We
will call $F(Z,X)$ the (standard) \emph{Effros fibred space }of the
fibration $p:Z\rightarrow X$.

\subsection{Analytic and Borel groupoids\label{Subsection: analytic and
Borel groupoids}}

A \emph{groupoid} $G$ is a small category where every arrow is invertible.
The set of objects of $G$ is denoted by $G^{0}$. We will regard $G^{0}$ as a
subset of $G$, by identifying an object with its identity arrow. Denote by $%
G^{2}$ the (closed) set of pairs of composable arrows%
\begin{equation*}
G^{2}=\left\{ \left( \gamma ,\rho \right) :s(\gamma )=r(\rho )\right\} \text{%
.}
\end{equation*}%
If $A,B$ are subsets of $G$, then $AB$ stands for the set%
\begin{equation*}
\left\{ \gamma \rho :\left( \gamma ,\rho \right) \in \left( A\times B\right)
\cap G^{2}\right\} \text{.}
\end{equation*}%
In particular if $Y\subset X$ then 
\begin{equation*}
YB=\left\{ \gamma \in B:r(\gamma )\in Y\right\}
\end{equation*}%
and 
\begin{equation*}
BY=\left\{ \gamma \in B:s(\gamma )\in Y\right\} \text{.}
\end{equation*}%
We write $xB$ for $\left\{ x\right\} B=r^{-1}\left[ \left\{ x\right\} \right]
\cap B$ and $Bx$ for $B\left\{ x\right\} =s^{-1}\left[ \left\{ x\right\} %
\right] \cap B$. If $A$ is a set of objects, then the \emph{restriction} $%
G_{|A}$ of $G$ to $A$ (this is called \textquotedblleft
contraction\textquotedblright\ in~\cite{mackey_ergodic_1963,
ramsay_virtual_1971}) is the groupoid 
\begin{equation*}
\left\{ \gamma \in G:s(\gamma )\in A,r(\gamma )\in A\right\}
\end{equation*}%
with set of objects $A$ and operations inherited from $G$.

To every groupoid $G$ one can associate the \emph{orbit equivalence relation}
$E_{G}$ on $G^{0}$ defined by $(x,y)\in E_{G}$ if and only if there is $%
\gamma \in G$ such that $s(\gamma )=x$ and $r(\gamma )=y$. The function%
\begin{eqnarray*}
G &\rightarrow &E_{G} \\
\gamma &\mapsto &(r(\gamma ),s(\gamma ))
\end{eqnarray*}%
is a continuous surjection. We say that a groupoid is \emph{principal }when
such a map is injective. Thus a principal groupoid is just an equivalence
relation on its set of objects. Conversely any equivalence relation can be
regarded as a principal groupoid.

The notion of \emph{functor }between groupoids is the usual notion from
category theory. Thus a functor from $G$ to $H$ is a function from $G$ to $H$
such that, for every $\gamma \in G$ and $\left( \rho _{0},\rho _{1}\right)
\in G^{2}$ the following holds:

\begin{itemize}
\item $F(s(\gamma ))=s(F(\gamma ))$;

\item $F(r(\gamma ))=r(F(\gamma ))$;

\item $F(\gamma ^{-1})=F(\gamma )^{-1}$;

\item $F(\rho _{0}\rho _{1})=F(\rho _{0})F(\rho _{1})$.
\end{itemize}

When $E$ and $E^{\prime }$ are principal groupoids, then functors from $E$
to $E^{\prime }$ are in 1:1 correspondence with \emph{reductions }from $E$
to $E^{\prime }$ in the sense of~\cite[Definition 5.1.1]{gao_invariant_2009}.

\begin{definition}
An \emph{analytic groupoid} is a groupoid endowed with an analytic Borel
structure making composition and inversion of arrows Borel and such that the
set of objects and, for every object $x$, the set of elements with source $x$%
, are standard Borel spaces with respect to the induced Borel structure. A 
\emph{(standard) Borel groupoid }is a groupoid endowed with a standard Borel
structure making composition and inversion of arrows Borel, and such that
the set of objects is a Borel subset.
\end{definition}

It is immediate to verify that principal analytic groupoids are precisely
analytic equivalence relations on standard Borel spaces. Similarly principal
Borel groupoids are precisely the Borel equivalence relations on standard
Borel spaces. A functor between analytic groupoids is Borel if it is Borel
as a function with respect to the given Borel structures.

\subsection{Polish groupoids and Polish groupoid actions\label{Subsection:
Polish groupoids and groupoid actions}}

\begin{definition}
A\emph{\ topological groupoid} is a groupoid endowed with a topology making
composition and inversion of arrows continuous.
\end{definition}

It is not difficult to see that for a topological groupoid the following
conditions are equivalent:

\begin{enumerate}
\item Composition of arrows is open;

\item The source map is open;

\item The range map is open.
\end{enumerate}

(See~\cite[Exercise I.1.8]{resende_lectures_2006}.)

\begin{definition}
A\emph{\ Polish groupoid} is a groupoid endowed with a locally Polish
topology such that

\begin{enumerate}
\item composition and inversion of arrows are continuous and open,

\item the set $G^{0}$ of objects is a Polish space with the subspace
topology,

\item for every $x\in G^{0}$ the sets $Gx$ and $xG$ are Polish spaces with
the subspace topology.
\end{enumerate}
\end{definition}

Polish groupoids have been introduced in~\cite{ramsay_mackey-glimm_1990}
with the extra assumption that the topology be regular or, equivalently,
globally Polish. It is nonetheless noticed in~\cite[page 362]%
{ramsay_mackey-glimm_1990} that one can safely dispense of this additional
assumption, without invalidating the results proved therein.

Suppose that $G$ is a Polish groupoid, and $X$ is a Polish space. A
continuous action of $G$ on $X$ is given by a continuous function $%
p:X\rightarrow G^{0}$ called \emph{anchor map} together with a continuous
function $\left( g,x\right) \mapsto gx$ from%
\begin{equation*}
G\ltimes X=\left\{ (\gamma ,x):p(x)=s(\gamma )\right\}
\end{equation*}%
to $X$ such that, for all $\gamma ,\rho \in G$ and $x\in X$

\begin{enumerate}
\item $\gamma (\rho x)=(\gamma \rho )x$,

\item $p(\gamma x)=r(\gamma )$, and

\item $p(x)x=x$.
\end{enumerate}

In such a case we say that $X$ is a Polish $G$-space. Similarly if $X$ is a
standard Borel space, then a Borel action of $G$ on $X$ is given by a Borel
map $p:X\rightarrow G^{0}$ together with a Borel map 
\begin{eqnarray*}
G\ltimes X &\rightarrow &X \\
(\gamma ,x) &\rightarrow &\gamma x
\end{eqnarray*}%
satisfying the same conditions as above. In this case $X$ will be called a
Borel $G$-space.

Clearly any Polish groupoid acts continuously on its space of objects $G^{0}$
by setting $p(x)=x$ and $(\gamma ,x)\mapsto r(\gamma )$. This will be called
the \emph{standard }action of $G$ on $G^{0}$.

Most of the usual notions for actions of groups, such as orbits, or
invariant sets, can be generalized in the obvious way to actions of
groupoids. If $X$ is a $G$-space, and $x\in X$, then its orbit $\left\{
\gamma x:s(\gamma )=p(x)\right\} $ is denoted by $\left[ x\right] $. The
orbit equivalence relation $E_{G}^{X}$ on $X$ is defined by $xE_{G}^{X}y$
iff $\left[ x\right] =\left[ y\right] $. If $A$ is a subset of $X$, then its
saturation 
\begin{equation*}
\left\{ \gamma a:a\in A\text{, }\gamma \in Gp(a)\right\}
\end{equation*}%
is denote by $\left[ A\right] $. An action is called \emph{free} if $\gamma
x=\rho x$ implies $\gamma =\rho $ for any $x\in X$ and $\gamma ,\rho \in Gp(x%
\mathbb{)}$.

Suppose that $G\ $is a Polish groupoid, and $X$ is a Borel $G$-space. If $%
x,y\in G^{0}$ are in the same orbit define the stabilizer 
\begin{equation*}
G_{x}=\left\{ \gamma \in G:s(\gamma )=p(x)\text{ and }\gamma x=x\right\}
\end{equation*}%
of $x$, and 
\begin{equation*}
G_{x,y}=\left\{ \gamma \in G:s(\gamma )=p(x)\text{ and }\gamma x=y\right\} 
\text{.}
\end{equation*}%
Observe that by~\cite[Theorem 9.17]{kechris_classical_1995} $G_{x}$ is a
closed subgroup of $p(x)Gp(x)$. Therefore $G_{x,y}$ is also closed, since $%
G_{x,y}=G_{x,x}\rho $ for any $\rho $ such that $s(\rho )=p(x)$ and $\rho
x=y $.

Suppose that $X$ and $Y$ are Borel $G$-spaces with anchor maps $p_{X}$ and $%
p_{Y}$. A Borel fibred map from $X$ to $Y$ is a Borel function $\varphi
:X\rightarrow Y$ such that $p_{Y}\circ \varphi =p_{X}$. A Borel fibred map
from $X$ to $Y$ is $G$-\emph{equivariant }if 
\begin{equation*}
\varphi (\gamma x)=\gamma \varphi (x)
\end{equation*}%
for $x\in X$ and $\gamma \in Gp(x)$. A\emph{\ Borel }$G$\emph{-embedding}
from $X$ to $Y$ is an injective $G$-equivariant Borel fibred map from $X$ to 
$Y$. Finally a \emph{Borel }$G$-\emph{isomorphism }from $X$ to $Y$ is a
Borel $G$-embedding which is also onto.

\subsection{Some examples of Borel groupoids\label{Subsection: examples}}

In this subsection we show how several natural categories of interest can be
endowed (after a suitable parametrization) with the structure of Borel
groupoid.

Let us first consider the category of \emph{complete separable metric spaces}%
, having surjective isometries as morphisms. This can be endowed with the
structure of Borel groupoid in the following way. Denote by $\mathbb{U}$ the
Urysohn universal metric space. (A survey about $\mathbb{U}$ and its
remarkable properties can be found in \cite{melleray_geometry_2007}.) Let $F(%
\mathbb{U)}$ be the Borel space of closed subsets of $\mathbb{U}$ endowed
with the Effros Borel structure. By universality of the Urysohn space, $F(%
\mathbb{U)}$ contains an isometric copy of any separable metric space.
Moreover any surjective isometry between closed subsets of $\mathbb{U}$ can
be identified with its graph, which is a closed subset of $\mathbb{U}\times 
\mathbb{U}$. The set $\mathbf{CMS}$ of such graphs is easily seen to be a
Borel subset of $F(\mathbb{U)}$. Moreover a standard computation shows that
composition and inversion of arrows are Borel functions in $\mathbf{CMS}$.
This shows that $\mathbf{CMS}$ is a Borel groupoid that can be seen as a
parametrization of the category of metric spaces with surjective isometries
as arrows.

More generally one can look at the category of \emph{separable }$\mathcal{L}$%
\emph{-structures} in some signature $\mathcal{L}$ of continuous logic. (A
complete introduction to continuous logic is \cite{ben_yaacov_model_2008}.)
One can identify any $\mathcal{L}$-structure with an $\mathcal{L}$-structure
having as support a closed subset of $\mathbb{U}$. In such case the
interpretation of a function symbol $f$ can be seen as a closed subset of $%
\mathbb{U}^{\left\vert f\right\vert +1}$ where $\left\vert f\right\vert $
denotes the arity of $f$. The interpretation of a relation symbol $B$ can be
seen as a closed subset of $\mathbb{U}^{\left\vert R\right\vert }\times 
\mathbb{R}$ where again $\left\vert R\right\vert $ denotes the arity of $B$.
(Here distances and relations are allowed to attain value in the whole real
line.) The set $\mathrm{Mod}(\mathcal{L}\mathbb{)}$ of such structures can
be verified to be a Borel subset of%
\begin{equation*}
F(\mathbb{U)}\times \prod_{f}F(\mathbb{U)}\times \prod_{B}F(\mathbb{U)}
\end{equation*}%
where $f$ and $B$ range over the function and relation symbols of $\mathcal{L%
}$. Similar parametrizations of the space of $\mathcal{L}$-structures can be
found in \cite{elliott_isomorphism_2013} and \cite%
{ben_yaacov_lopez-escobar_2014}. As before the space $\mathbf{\mathrm{Mod}}(%
\mathcal{L}\mathbb{)}$ of isomorphisms between $\mathcal{L}$-structures
(identified with their graph) is a Borel subset of $F(\mathbb{U)}$, and
composition and inversion of arrows are Borel maps. Thus one can regard $%
\mathbf{\mathrm{Mod}}(\mathcal{L}\mathbb{)}$ as the Borel groupoid of $%
\mathcal{L}$-structures. In the particular case when one considers \emph{%
discrete }structures then one can replace the Urysohn space with $\omega $.

As a particular case of separable structures in a given signature one can
consider \emph{separable C*-algebras}. (The book \cite%
{blackadar_operator_2006} is a complete reference for the theory of operator
algebras.) The complexity of the classification problem for separable
C*-algebras has recently attracted considerable interest; see \cite%
{farah_turbulence_2014, farah_descriptive_2012, elliott_isomorphism_2013,
sabok_completeness_2013}. Particularly important classes for the
classification program are \emph{nuclear} and \emph{exact} C*-algebras; see 
\cite[Section IV.3]{blackadar_operator_2006}. Separable exact C*-algebras
are precisely the closed self-adjoint subalgebras of the Cuntz algebra $%
\mathcal{O}_{2}$ \cite{kirchberg_embedding_2000}. Thus the Borel groupoid $%
\mathbf{C}^{\ast }\mathbf{Exact}$ of closed subalgebras of $\mathcal{O}_{2}$%
---where a *-isomorphism between closed subalgebras is identified with its
graph---can be regarded as a parametrization for the category of exact
C*-algebras having *-isomorphisms as arrows. The category of (simple,
unital) nuclear C*-algebras can be regarded as the restriction of $\mathbf{C}%
^{\ast }\mathbf{Exact}$ to the Borel set of (simple, unital) self-adjoint
subalgebras of $\mathcal{O}_{2}$; see \cite[Section 7]{farah_turbulence_2014}%
.

We now look at the category of \emph{Polish groups} with continuous group
isomorphisms as arrows. Denote by $\mathrm{Iso}(\mathbb{U)}$ the group of
isometries of the Urysohn space endowed with the topology of pointwise
convergence. Recall that $\mathrm{Iso}(\mathbb{U)}$ is a universal Polish
group \cite{uspenskij_group_1990}, i.e. it contains any other Polish group
as closed subgroup. The space $SG\left( \mathrm{Iso}(\mathbb{U)}\right) $ of
closed subgroups of $\mathrm{Iso}(\mathbb{U)}$ endowed with the Effros Borel
structure can be regarded as the standard Borel space of Polish groups.
Moreover a continuous isomorphism between closed subgroups of $\mathrm{Iso}(%
\mathbb{U)}$ can be identified with its graph, which is a closed subgroup of 
$\mathrm{Iso}(\mathbb{U)}\times \mathrm{Iso}(\mathbb{U)}$.\ It is not
difficult to check that the set $\mathbf{PG}$ of such closed subgroups of $%
\mathrm{Iso}(\mathbb{U)}\times \mathrm{Iso}(\mathbb{U)}$ is a Borel subset
of the space $SG\left( \mathrm{Iso}(\mathbb{U)}\times \mathrm{Iso}(\mathbb{U)%
}\right) $ of closed subgroups of $\mathrm{Iso}(\mathbb{U)}\times \mathrm{Iso%
}(\mathbb{U)}$ endowed with the Effros Borel structure. (Fix a countable
neighborhood basis $\mathcal{N}$ of the identity in $\mathrm{\mathrm{Iso}}(%
\mathbb{U)}$, and observe that a closed subgroup $H$ of $\mathrm{Iso}(%
\mathbb{U)}\times \mathrm{Iso}(\mathbb{U)}$ is in $\mathbf{PG}$ if and only
if $\forall U\in \mathcal{N}$ $\exists V\in \mathcal{N}$ such that $H\cap
\left( U\times \left( \mathrm{Iso}(\mathbb{U)}\backslash cl(V)\right)
\right) =\varnothing $ and $H\cap \left( \left( \mathrm{Iso}(\mathbb{U)}%
\backslash cl(V)\right) \times U\right) =\varnothing $.) Moreover a standard
calculation shows that composition and inversion of arrows are Borel
functions in $\mathbf{PG}$. (For composition of arrows, observe that if as
before $\mathcal{N}$ is a countable basis of neighborhoods of the identity
in $\mathrm{Iso}(\mathbb{U)}$, $D$ is a dense subset of $\mathrm{Iso}(%
\mathbb{U)}$, $A$ and $B$ are open subsets of $\mathrm{Iso}(\mathbb{U)}$,
and $\varphi ,\psi \in \mathbf{PG}$, then $\left( \varphi \circ \psi \right)
\cap \left( A\times B\right) \neq \varnothing $ if and only if there are $%
U,V\in \mathcal{N}$ and $g,h\in D$ with $cl(V)^{2}h\subset B$ and $Ug\subset
A$ such that $\psi \cap \left( U^{2}\times \left( \mathrm{Iso}(\mathbb{U)}%
\backslash cl(V)\right) \right) =\varnothing $, $\varphi \cap \left( A\times
Ug\right) \neq \varnothing $, and $\psi \cap \left( Ug\times Vh\right) \neq
\varnothing $.) This shows that $\mathbf{PG}$ is a Borel groupoid that can
be seen as a parametrization of the category of Polish groups with
continuous isomorphisms as arrows.

A similar discussion applies to the category of separable Banach spaces with
linear (not necessarily isometric) isomorphisms as arrows. In this case one
considers a universal separable Banach space, such as $C\left[ 0,1\right] $.
One then looks at the standard Borel space of closed subspaces of $C\left[
0,1\right] $ as set of objects, and the set of closed subspaces of $C\left[
0,1\right] \oplus C\left[ 0,1\right] $ that code a linear isomorphism
between closed subspaces of $C\left[ 0,1\right] $ as set of arrows. The
proof that these sets are Borel with respect to the Effros Borel structure
is analogous to the case of Polish groups.

\subsection{The action groupoid\label{Subsection: the action groupoid}}

Suppose that $G$ is a Polish groupoid, and $X$ is a Polish $G$-space.
Consider the groupoid 
\begin{equation*}
G\ltimes X=\left\{ (\gamma ,x)\in G\times X:s(\gamma )=p(x)\right\} \text{,}
\end{equation*}%
where composition and inversion of arrows are defined by%
\begin{equation*}
\left( \rho ,\gamma x\right) (\gamma ,x)=\left( \rho \gamma ,x\right)
\end{equation*}%
and%
\begin{equation*}
(\gamma ,x)^{-1}=\left( \gamma ^{-1},\gamma x\right) \text{.}
\end{equation*}%
The set of objects of $G\ltimes X$ is%
\begin{equation*}
G^{0}\ltimes X=\left\{ \left( a,x\right) \in G^{0}\times X:p(x)=a\right\} 
\text{.}
\end{equation*}%
Endow $G\ltimes X$ with the subspace topology from $G\times X$. Observe that
the function%
\begin{eqnarray*}
X &\rightarrow &G^{0}\ltimes X \\
x &\mapsto &\left( p(x),x\right)
\end{eqnarray*}
is a homeomorphism from $X$ to the set of objects of $G\ltimes X$. We can
therefore identify the latter with $X$. Under this identification the source
of $(\gamma ,x)$ is $x$ and the range is $\gamma x$. We claim that $G\ltimes
X$ is a Polish groupoid, called the \emph{action groupoid }associated with
the Polish $G$-space $X$. Clearly the topology is locally Polish, and
composition and inversion of arrows are continuous. We need to show that the
source map is open. Suppose that $V$ is an open subset of $G$, $U$ is an
open subset of $X$, and $W$ is the open subset%
\begin{equation*}
\left\{ (\gamma ,x):\gamma \in V,x\in U\right\}
\end{equation*}%
of $G\ltimes X$. Suppose that $W$ is nonempty and pick $\left( \gamma
_{0},x_{0}\right) \in W$. Thus $x_{0}\in U$ and $p(x_{0})=s(\gamma _{0})\in s%
\left[ V\right] $. Therefore there is an open subset $U_{0}$ of $U$
containing $x_{0}$ such that $p\left[ U_{0}\right] \subset s\left[ V\right] $%
. We claim now that $U_{0}$ is contained in the image of $W$ under the
source map. In fact if $x\in U_{0}$ then $p(x)=s(\gamma )$ for some $\gamma
\in V$ and therefore $x$ is the source of the arrow $(\gamma ,x)$ in $W$.
This concludes the proof of the fact that $G\ltimes X$ is a Polish groupoid.
To summarize we can state the following proposition.

\begin{proposition}
\label{Proposition: action groupoid}Suppose that $G$ is a Polish groupoid,
and $X$ is a Polish $G$-space. The action groupoid $G\ltimes X$ as defined
above is a Polish groupoid. Moreover the map 
\begin{eqnarray*}
X &\rightarrow &\left( G\ltimes X\right) ^{0} \\
x &\mapsto &\left( p(x),x\right)
\end{eqnarray*}%
is a homeomorphism such that, for every $x,x^{\prime }\in X$, 
\begin{equation*}
xE_{G}^{X}x^{\prime }\quad \text{iff}\quad \left( p(x),x\right) E_{G\ltimes
X}\left( p(x^{\prime }),x^{\prime }\right) \text{.}
\end{equation*}
\end{proposition}

\subsection{Functorial reducibility\label{Subsection: functorial
reducibility}}

\begin{definition}
\label{Definition: Borel reduction}Suppose that $G$ and $H$ are analytic
groupoids. A\emph{\ Borel reduction} from $G$ to $H$ is a Borel functor $F$
from $G$ to $H$ such that $xGy\neq \varnothing $ whenever $F(x)H{}F(y)\neq
\varnothing $.
\end{definition}

Equivalently a Borel functor $F$ from $G$ to $H$ is a Borel reduction from $%
G $ to $H$ when the function%
\begin{eqnarray*}
G^{0} &\rightarrow &H^{0} \\
x &\mapsto &F(x)
\end{eqnarray*}%
is a Borel reduction from $E_{G}$ to $E_{H}$ in the sense of~\cite[%
Definition 5.1.1]{gao_invariant_2009}.

\begin{definition}
\label{Definition: Borel reducibility}Suppose that $G$ and $H$ are analytic
groupoids. We say that $G$ is\emph{\ Borel reducible} to $H$---in formulas $%
G\leq _{B}H$---if there is a Borel reduction from $G$ to $H$.
\end{definition}

The notion of bireducibility is defined accordingly.

\begin{definition}
\label{Definition: Borel bireducibility}Suppose that $G$ and $H$ are
analytic groupoids. We say that $G$ is\emph{\ Borel bireducible} to $H$---in
formulas $G\sim _{B}H$---if $G$ is Borel reducible to $H$ and vice versa.
\end{definition}

When $E$ and $E^{\prime }$ are principal analytic groupoids, then the Borel
reductions\emph{\ }from $E$ to $E^{\prime }$ are in 1:1 correspondence with
Borel reductions from $E$ to $E^{\prime }$ in the usual sense of Borel
complexity theory; see~\cite[Definition 5.1.1]{gao_invariant_2009}. In
particular Definition~\ref{Definition: Borel reducibility} \emph{generalizes}
the notion of Borel reducibility from analytic equivalence relations to
analytic groupoids.

Similarly as in the case of reducibility for equivalence relations, one can
impose further requirements on the reduction map. If $G$ and $H$ are
analytic groupoid, we say that $G$ is injectively Borel reducible to $H$%
---in formulas $G\sqsubseteq _{B}H$ if there is an \emph{injective }Borel
reduction from $G$ to $H$. When $G$ and $H$ are Polish groupoid, one can
also insist that the reduction be continuous rather than Borel. One then
obtains the notion of continuous reducibility $\leq _{c}$ and continuous
injective reducibility $\sqsubseteq _{c}$.

Definition \ref{Definition: Borel reducibility} provides a natural notion of
comparison between analytic groupoids. This allows one to build a hierarchy
of complexity of analytic groupoids, that includes the usual hierarchy of
Borel equivalence relations. The \emph{functorial Borel complexity }of an
analytic groupoid will denote the position of the given groupoid in such a
hierarchy.

\subsection{Category preserving maps\label{Subsection: Category preserving
maps}}

According to~\cite[Definition A.2]{melleray_generic_2013} a continuous map $%
f:X\rightarrow Y$ between Polish spaces is \emph{category preserving }if for
any comeager subset $C$ of $Y$ the inverse image $f^{-1}\left[ C\right] $ of 
$C$ under $f$ is a comeager subset of $X$. It is not difficult to see that
any continuous open map is category preserving~\cite[Proposition A.3]%
{melleray_generic_2013}.

Category-preserving maps satisfy a suitable version of the
classical\linebreak Kuratowski-Ulam theorem for coordinate projections. We
will state the particular case of this result for continuous open maps in
the following lemma, which is Theorem~A.1 in~\cite{melleray_generic_2013}.

\begin{lemma}
\label{Lemma: Kuratowski-Ulam}Suppose that $X$ is second countable space, $Y$
is a Baire space, and $f:X\rightarrow Y$ is an open continuous map such that 
$f^{-1}\left\{ y\right\} $ is a Baire space for every $y\in Y$. If $A\subset
X$ has the Baire property, then the following statements are equivalent:

\begin{enumerate}
\item $A$ is comeager;

\item $\forall ^{\ast }y\in Y$, $A\cap f^{-1}\left\{ y\right\} $ is comeager
in $f^{-1}\left\{ y\right\} $.
\end{enumerate}
\end{lemma}

\subsection{Vaught transforms\label{Subsection: Vaught transform}}

Suppose in the following that $G$ is a Polish groupoid,%
\begin{equation*}
\mathcal{A}=\left\{ U_{n}:n\in \omega \right\}
\end{equation*}%
is a basis of Polish open subsets of $G$, and $X$ is a Borel $G$-space.

\begin{definition}
For $A\subset X$ and $V\subset G$, define the \emph{Vaught transforms}%
\begin{equation*}
A^{\bigtriangleup V}=\left\{ x\in X:Vp(x)\neq \varnothing \text{ and }%
\exists ^{\ast }\gamma \in Vp(x)\text{, }\gamma x\in A\right\}
\end{equation*}%
and%
\begin{equation*}
A^{\ast V}=\left\{ x\in X:Vp(x)\neq \varnothing \text{ and }\forall ^{\ast
}\gamma \in Vp(x)\text{, }\gamma x\in A\right\} \text{.}
\end{equation*}
\end{definition}

In the particular case when $G$ is a Polish group, and $X$ is a Borel $G$%
-space, this definition coincide with the usual Vaught transform; cf.~\cite[%
Definition 3.2.2]{gao_invariant_2009}.

\begin{lemma}
\label{Lemma: basic Vaught}Assume that $B$ and $A_{n}$ for $n\in \omega $
are subsets of $X$. If $V$ is an open subset of $G$, then the following hold:

\begin{enumerate}
\item $B^{\bigtriangleup G}$ and $B^{\ast G}$ are invariant subsets of $X$;

\item $\left( \bigcap_{n}A_{n}\right) ^{\ast V}=\bigcap_{n}A_{n}^{\ast V}$;

\item $\left( \bigcup_{n}A_{n}\right) ^{\bigtriangleup
V}=\bigcup_{n}A_{n}^{\bigtriangleup V}$;

\item $p^{-1}\left[ s\left[ V\right] \right] $ is the disjoint union of $%
\left( X\left\backslash B\right. \right) ^{\ast V}$ and $B^{\bigtriangleup
V} $;

\item If $B$ is analytic, then $B^{\bigtriangleup V}=\bigcup \left\{ B^{\ast
U}:V\supset U\in \mathcal{A}\right\} $ and $B^{\ast V}=\bigcap \left\{
B^{\bigtriangleup U}:V\supset U\in \mathcal{A}\right\} $.
\end{enumerate}
\end{lemma}

Lemma~\ref{Lemma: basic Vaught} is elementary and can be proved similarly as 
\cite[Proposition 3.2.5]{gao_invariant_2009}.

\begin{lemma}
\label{Lemma: shift Vaught}Suppose that $B\subset X$ is analytic, and $%
U\subset G$ is open. If $x\in X$ and $\gamma \in Gp(x)$, then the following
statements are equivalent:

\begin{enumerate}
\item $\gamma x\in B^{\bigtriangleup U}$;

\item $x\in B^{\ast V}$ for some $V\in \mathcal{A}$ such that $V\gamma
^{-1}\subset Ur(\gamma )$;

\item $x\in B^{\bigtriangleup V}$ for some $V\in \mathcal{A}$ such that $%
V\gamma ^{-1}\subset Ur(\gamma )$;

\item there are $V,W\in \mathcal{A}$ such that $VW^{-1}\subset U$, $\gamma
\in W$, and $x\in B^{\bigtriangleup V}$.
\end{enumerate}
\end{lemma}

\begin{description}
\item[(1)$\Rightarrow $(2)] By hypothesis $Ur(\gamma )\neq \varnothing \,$%
and $\exists ^{\ast }\rho \in Ur(\gamma )$ such that $\rho \gamma x\in B$.
Therefore $U\gamma \neq \varnothing $ and $\exists ^{\ast }\rho \in U\gamma $
such that $\rho x\in B$. Since $B$ is analytic and the action is Borel, the
set%
\begin{equation*}
\left\{ \rho \in U\gamma :\rho x\in B\right\}
\end{equation*}%
is analytic and in particular it has the Baire property. It follows that
there is $V\in \mathcal{A}$ such that $Vp(x)\neq \varnothing $, $%
Vp(x)\subset U\gamma $, and $\forall ^{\ast }\rho \in V$, $\rho x\in B$.
Observe that $V\gamma ^{-1}\subset Ur(\gamma )$.

\item[(2)$\Rightarrow $(3)] Obvious.

\item[(3)$\Rightarrow $(1)] Observe that $\varnothing \neq Vp(x)\subset
U\gamma $. Thus $U\gamma \neq \varnothing $ and $\exists ^{\ast }\rho \in
U\gamma $ such that $\rho x\in B$. Thus $Up\left( \gamma z\right) \neq
\varnothing $ and $\exists ^{\ast }\rho \in Up(\gamma x)$, $\rho \gamma x\in
B$. This shows that $\gamma x\in B^{\bigtriangleup U}$.

\item[(2)$\Rightarrow $(4)] Pick $v\in Vp(x)$ and observe that $v\gamma
^{-1}\in Ur(\gamma )$. Therefore there are $W,V_{0}\in \mathcal{A}$ such
that $v\in V_{0}\subset V$, $\gamma \in W$, and $V_{0}W^{-1}\subset U$.
Moreover since $x\in B^{\ast V}$ and $V_{0}\subset V$ we have that $x\in
B^{\ast V_{0}}$.

\item[(4)$\Rightarrow $(2)] Obvious.
\end{description}

If $A$ is a subset of $G\ltimes X$ and $x\in X$, then $A_{x}$ denotes the $x$%
-fiber $\left\{ \gamma \in G:(\gamma ,x)\in A\right\} $ of $A$. The proof of
the following lemma is inspired by the proof of the Montgomery-Novikov
theorem; see~\cite[Theorem 16.1]{kechris_classical_1995}.

\begin{lemma}
\label{Lemma: Montgomery-Novikov}If $A$ is a Borel subset of $G\ltimes X$
and $V\subset G$ is open, then%
\begin{equation*}
\left\{ x\in X:Vp(x)\neq \varnothing \text{ and }A_{x}\text{ is nonmeager in 
}Vp(x)\right\}
\end{equation*}%
is Borel. The same conclusion holds if one replaces ``nonmeager'' with
``comeager'' or ``meager''.
\end{lemma}

\begin{proof}
Define $\mathcal{E}$ to be the class subsets of subsets $A$ of $G\ltimes X$
such that 
\begin{equation*}
\left\{ x\in X:Vp(x)\neq \varnothing \text{ and }A_{x}\text{ is nonmeager in 
}Vp(x)\right\}
\end{equation*}%
is Borel for every nonempty open subset $V$ of $G$. We claim that:

\begin{enumerate}
\item $\mathcal{E}$ contains the sets of the form%
\begin{equation*}
U\ltimes B=\left\{ \left( \rho ,x\right) \in G\ltimes X:x\in B\text{, }\rho
\in U\right\}
\end{equation*}%
for $B\subset X$ Borel and $U\subset G$ open;

\item $\mathcal{E}$ is closed by taking countable unions;

\item $\mathcal{E}$ is closed by taking complements.
\end{enumerate}

In fact:

\begin{enumerate}
\item If $A=U\ltimes B$ where $B\subset X$ is Borel and $U\subset G$ is open
then for every nonempty open set $V$%
\begin{eqnarray*}
&&\left\{ x\in X:Vp(x)\neq \varnothing \text{ and }A_{x}\text{ is nonmeager
in }Vp(x)\right\} \\
&=&B\cap p^{-1}\left[ s\left[ U\cap V\right] \right] \text{;}
\end{eqnarray*}%
is Borel.

\item If $A=\bigcup_{n}A_{n}$ then for every nonempty open set%
\begin{eqnarray*}
&&\left\{ x\in X:Vp(x)\neq \varnothing \text{ and }A_{x}\text{ is nonmeager
in }Vp(x)\right\} \\
&=&\bigcup_{n\in \omega }\left\{ x\in X:Vp(x)\neq \varnothing \text{ and }%
\left( A_{n}\right) _{x}\text{ is nonmeager in }Vp(x)\right\} \text{;}
\end{eqnarray*}

\item If $A\subset G\ltimes X$ then for every nonempty open set $V$%
\begin{eqnarray*}
&&\left\{ x\in X:Vp(x)\neq \varnothing \text{ and }\left( \left( G\ltimes
X\right) \backslash A\right) _{x}\text{ is nonmeager in }V\right\} \\
&=&\left\{ x\in X:Vp(x)\neq \varnothing \text{ and }A_{x}\text{ is not
comeager in }V\right\} \\
&=&\bigcup_{U_{n}\subset V}\left\{ x\in X:U_{n}p(x)\neq \varnothing \text{
and }A_{x}\text{ is meager in }U_{n}\right\} \\
&=&\bigcup_{U_{n}\subset V}\left( p^{-1}s\left[ U_{n}\right] \backslash
\left\{ x\in X:U_{n}p(x)\neq \varnothing \text{ and }A_{x}\text{ is
nonmeager in }U_{n}\right\} \right) \text{.}
\end{eqnarray*}
\end{enumerate}

A similar argument shows that the same conclusion holds after replacing
``nonmeager'' with ``meager'' or ``comeager''.
\end{proof}

\begin{lemma}
\label{Lemma: Borel Vaught transform}If $A\subset X$ is Borel and $V\subset
G $ is open then $A^{\bigtriangleup V}$ and $A^{\ast V}$ are Borel.
\end{lemma}

\begin{proof}
Consider the subset%
\begin{equation*}
\widetilde{A}=\left\{ \left( \rho ,x\right) \in G\ltimes X:\rho x\in
A\right\}
\end{equation*}%
and observe that $\widetilde{A}$ is a Borel subset of $G\ltimes X$ such that%
\begin{equation*}
A^{\bigtriangleup V}=\left\{ x\in X:Vp(x)\neq \varnothing \text{ and }%
\widetilde{A}_{x}\text{ is nonmeger in }Vp(x)\right\}
\end{equation*}%
and%
\begin{equation*}
A^{\ast V}=\left\{ x\in X:Vp(x)\neq \varnothing \text{ and }\widetilde{A}_{x}%
\text{ is comeager in }Vp(x)\right\} \text{.}
\end{equation*}%
The conclusion now follows from Lemma~\ref{Lemma: Montgomery-Novikov}.
\end{proof}

\begin{lemma}
\label{Lemma: Borel class Vaught transform}Assume that $X$ is a Polish $G$%
-space. If $B\subset X$, $U\subset G$ is open, and $\alpha \in \omega _{1}$,
then the following hold:

\begin{enumerate}
\item If $B$ is open, then $B^{\bigtriangleup U}$ is open;

\item If $B$ is $\mathbf{\Sigma }_{\alpha }^{0}$, then $B^{\bigtriangleup U}$
is $\mathbf{\Sigma }_{\alpha }^{0}$ relatively to $p^{-1}s\left[ U\right] $;

\item If $B$ is $\mathbf{\Pi }_{\alpha }^{0}$, then $B^{\ast U}$ is $\mathbf{%
\Pi }_{\alpha }^{0}$ relatively to $p^{-1}s\left[ U\right] $.
\end{enumerate}
\end{lemma}

\begin{proof}
The proof is analogous to the corresponding one for group actions; see~\cite[%
Theorem 3.2.9]{gao_invariant_2009}. Suppose that $B$ is open, and pick $x\in
B^{\bigtriangleup U}$. Thus $Up(x)\neq \varnothing $ and $\exists ^{\ast
}\rho \in Up(x)$ such that $\rho x\in B$. Pick $U_{0}\subset U$ open such
that $x\in B^{\ast U_{0}}$ and $\rho \in U_{0}p(x)$ such that $\rho x\in B$.
Since $B$ is open and the action is continuous there are open subsets $W$
and $V$ containing $x$ and $\rho $ such that $V\subset U_{0}$, $VW\subset B$%
, and $p\left[ W\right] \subset s\left[ V\right] $. We claim that $W\subset
B^{\bigtriangleup U}$. In fact if $w\in W$ then $Vp(w)\neq \varnothing $.
Moreover since $VW\subset B$, $\exists ^{\ast }\rho \in Vp(w)$ such that $%
\rho w\in B$. This concludes the proof that $B^{\bigtriangleup U}$ is open.
The other statements follow via (2),(3), and (4) of Lemma~\ref{Lemma: basic
Vaught}.
\end{proof}

Using the Vaught transform it is easy to see that, if $X$ is a Borel $G$%
-space, then the orbit equivalence relation $E_{G}^{X}$ is idealistic. (This
is well known when $G$ is a Polish group; cf.~\cite[Proposition 5.4.10]%
{gao_invariant_2009}.) Recall that an equivalence relation $E$ on a standard
Borel space $X$ is \emph{idealistic} if there is a map $\left[ x\right]
_{E}\mapsto I_{\left[ x\right] _{E}}$ assigning to each equivalence class $%
\left[ x\right] _{E}$ of $E$ and ideal $I_{\left[ x\right] _{E}}$ of subsets
of $\left[ x\right] _{E}$ such that $\left[ x\right] _{E}\notin I_{\left[ x%
\right] _{E}}$, and for every Borel subset $A$ of $X\times X$ the set $A_{I}$
defined by $x\in A_{I}$ iff $\left\{ y\in \left[ x\right] _{E}:(x,y)\in
A\right\} \in I_{\left[ x\right] _{E}}$ is Borel; see~\cite[Definition 5.4.9]%
{gao_invariant_2009}.

\begin{proposition}
\label{Proposition: idealistic}If $X$ is a Borel $G$-space, then the orbit
equivalence relation $E_{G}^{X}$ is idealistic.
\end{proposition}

\begin{proof}
Pick $x\in X$ and denote by $C$ the orbit of $x$. Define the ideal $I_{C}$
of subsets of $C$ by $S\in I_{C}$ iff $\forall ^{\ast }\rho \in Gp(x)$, $%
r(\rho )\notin S$. Observe that this does not depend from the choice of $x$.
In fact suppose that $y\in C$ and hence $y=\gamma x$ for some $\gamma \in
Gp(x)$. Assume moreover that $S\subset C$ is such that $\forall ^{\ast }\rho
\in Gp(x)$, $r(\rho )\notin S$. Consider the homeomorphism $\Phi $ from $%
Gp(x)$ to $Gp(y)$ given by $\rho \mapsto \rho \gamma $. It is apparent that 
\begin{equation*}
\Phi \left[ \left\{ \rho \in Gp(x):r(\rho )\notin S\right\} \right] =\left\{
\rho \in Gp(y):r(\rho )\notin S\right\} \text{.}
\end{equation*}%
This shows that $\forall ^{\ast }\rho \in Gp(y)$, $r(\rho )\notin S$, and
hence the definition of $I_{C}$ is does not depend from the choice of $x\in
C $. Clearly $C\notin I_{C}$ since $Gp(x)$ is a Baire space. It is not
difficult to verify that $I_{C}$ is a $\sigma $-ideal. Suppose that $%
A\subset X\times X$ is Borel, and consider the set $A_{I}$ defined by $x\in
A_{I}$ iff $\left\{ y\in \left[ x\right] :(x,y)\in A\right\} \in I_{\left[ x%
\right] }$. Observe that $x\in A_{I}$ iff $\forall ^{\ast }\rho \in Gp(x)$, $%
\left( x,r(\rho )\right) \notin A$. Consider 
\begin{equation*}
X\ast X=\left\{ (x,y)\in X\times X:p(x)=p(y)\right\}
\end{equation*}%
and the action of $G$ on $X\ast X$ defined by $p(x,y)=p(x)=p(y)$ and $\gamma
(x,y)=\left( x,\gamma y\right) $ for $\gamma \in Gp(x,y)$. Observe that $%
x\in A_{I}$ if and only if $(x,x)\in \left( \left( X\ast X\right)
\left\backslash A\right. \right) ^{\ast G}$. This shows that $A_{I}$ is
Borel by Lemma~\ref{Lemma: Borel Vaught transform}.
\end{proof}

Let us denote as customary by $E_{1}$ the tail equivalence relation for
sequences in $\left[ 0,1\right] $. If $E$ is an idealistic Borel equivalence
relation, then $E_{1}$ is not Borel reducible to $E$ by~\cite[Theorem 4.1]%
{kechris_classification_1997}. Therefore we obtain from Proposition~\ref%
{Proposition: idealistic} the following corollary:

\begin{corollary}
\label{Corollary: not reduce E1}If $X$ is a Borel $G$-space with Borel orbit
equivalence relations, then the orbit equivalence relation $E_{1}$ is not
Borel reducible to $E_{G}^{X}$.
\end{corollary}

Corollary~\ref{Corollary: not reduce E1} holds more generally for arbitrary
Borel $G$-spaces, with not necessarily Borel orbit equivalence relations.
This was shown by the present author in collaboration with Samuel Coskey,
George Elliott, and Ilijas Farah by adapting the proof of~\cite[Chapter 8]%
{hjorth_classification_2000}.

An equivalence relation $E$ on a standard Borel space $E$ is \emph{smooth}
if it is Borel reducible to the relation of equality in some Polish space~%
\cite[Definition 5.4.1]{gao_invariant_2009}. By~\cite[Theorem 5.4.11]%
{gao_invariant_2009} an equivalence relation has a Borel selector precisely
when it is smooth and idealistic. Therefore the following corollary follows
immediately from Proposition~\ref{Proposition: idealistic}.

\begin{corollary}
\label{Corollary: Borel selector}If $X$ is a Polish $G$-space such that $%
E_{G}^{X}$ is smooth, then $E_{G}^{X}$ has a Borel selector.
\end{corollary}

\begin{corollary}
\label{Corollary: functorial smooth}If $G$ and $H$ are Polish groupoids such
that $E_{G}$ and $E_{H}$ are smooth, then $G\leq _{B}H$ if and only if $%
E_{G}\leq E_{H}$.
\end{corollary}

\subsection{Borel orbits}

We now observe that, if $G$ is a Polish groupoid, then the orbits of any
Polish $G$-space are Borel.

\begin{proposition}
\label{Proposition: Borel classes}If $G$ is a Polish groupoid, and $X$ is a
Polish $G$-space, then the orbit equivalence relation $E_{G}^{X}$ is
analytic and has Borel classes.
\end{proposition}

\begin{proof}
By Proposition~\ref{Proposition: action groupoid} we can consider without
loss of generality the case of the standard action of $G$ on its set of
objects $G^{0}$. Fix $x\in G^{0}$ and consider the right action of $xGx$ on $%
Gx$ by composition. Observe that $xGx$ is a Polish group, and $Gx$ is a
right Polish $xGx$-space with closed orbits.\ Therefore by~\cite[Proposition
3.4.6]{gao_invariant_2009} the corresponding orbit equivalence relation $%
E_{xGx}^{Gx}$ has a Borel transversal $T$. The orbit $\left[ x\right] $ is
the image of $T$ under the range map $r$. Since $r$ is 1:1 on $T$, it
follows that $\left[ x\right] $ is Borel by~\cite[Theorem 15.1]%
{kechris_classical_1995}. Observe now that the orbit equivalence relation $%
E_{G}$ is the image of the standard Borel space $G$ under the Borel function 
$\gamma \mapsto (r(\gamma ),s(\gamma ))$. This shows that $E_{G}$ is
analytic.
\end{proof}

Similarly as in the case of Polish group actions, a uniform bound on the
complexity of the orbits in the Borel hierarchy entails Borelness of the
orbit equivalence relation.

\begin{theorem}
\label{Theorem: Sami}Suppose that $G$ is Polish groupoid, and $X$ is a
Polish $G$-space. The orbit equivalence relation $E_{G}^{X}$ is Borel if and
only if there is $\alpha \in \omega _{1}$ such that every orbit is $\mathbf{%
\Pi }_{\alpha }^{0}$
\end{theorem}

\begin{proof}
One direction is obvious. For the other one consider for $\alpha \in \omega
_{1}$ the relation $E_{\alpha }$ of $X$ defined by%
\begin{equation*}
(x,y)\in E_{\alpha }
\end{equation*}%
iff for every $G$-invariant $\mathbf{\Pi }_{\alpha }^{0}$ set $W\subset X$
we have that $x\in W$ iff $y\in W$. If every orbit is $\mathbf{\Pi }_{\alpha
}^{0}$ then $E_{G}^{X}=E_{\alpha }$. It is thus enough to prove that $%
E_{\alpha }$ is co-analytic for every $\alpha \in \omega _{1}$. Consider a
universal $\mathbf{\Pi }_{\alpha }^{0}$ subset $U$ of $\omega ^{\omega
}\times X$. Define the action of $G$ on $\omega ^{\omega }\times X$ by
setting $p\left( a,b\right) =p(b)$ and $\gamma \left( a,b\right) =\left(
a,\gamma b\right) $. Define now 
\begin{equation*}
T=U^{\ast G}
\end{equation*}%
and observe that $T$ is $\mathbf{\Pi }_{\alpha }^{0}$ since $U$ is $\mathbf{%
\Pi }_{\alpha }^{0}$. Denote by $T_{a}$ the section%
\begin{equation*}
\left\{ b\in X:\left( a,b\right) \in T\right\}
\end{equation*}%
for $a\in \omega ^{\omega }$. We have that 
\begin{eqnarray*}
b \in T_{a}&\Leftrightarrow &\left( a,b\right) \in T \\
&\Leftrightarrow &\forall ^{\ast }\gamma \in Gp(b)\text{, }\left( a,\gamma
b\right) \in U \\
&\Leftrightarrow &\forall ^{\ast }\gamma \in Gp(b),\gamma b\in U_{a} \\
&\Leftrightarrow &b\in \left( U_{a}\right) ^{\ast G}\text{.}
\end{eqnarray*}%
This shows that $T_{a}$ is a $G$-invariant $\mathbf{\Pi }_{\alpha }^{0}$
subset of $X$ for every $\alpha \in \omega ^{\omega }$. Conversely if $A$ is
a $G$-invariant $\mathbf{\Pi }_{\alpha }^{0}$ set then $A=U_{a}$ for some $%
a\in \omega ^{\omega }$ and hence%
\begin{equation*}
A=A^{\ast G}=\left( U_{a}\right) ^{\ast G}=T_{a}\text{.}
\end{equation*}%
This shows that $\left\{ T_{a}:a\in \omega ^{\omega }\right\} $ is the
collection of all invariant $\mathbf{\Pi }_{\alpha }^{0}$ sets. It follows
that $(x,y)\in E_{\alpha }$ iff $\forall a\in \omega ^{\omega }$, $\left(
a,x\right) \in T$. Therefore $E_{\alpha }$ is co-analytic.
\end{proof}

Theorem~\ref{Theorem: Sami} was proved for Polish group actions in~\cite[%
Sections 3.6 and 3.7]{sami_polish_1994}.

\section{Effros' theorem and the Glimm-Effros dichotomy\label{Section:
Glimm-Effros}}

\subsection{Effros' theorem}

\begin{lemma}
\label{Lemma: meager transform}Suppose that $G$ is a Polish groupoid.
Consider the standard action of $G$ on $G^{0}$, and the corresponding Vaught
transform. If $A\subset G^{0}$ is meager, then $A^{\bigtriangleup G}$ is
meager.
\end{lemma}

\begin{proof}
The source map $r:G\rightarrow G^{0}$ is open and, in particular, category
preserving; see Subsection~\ref{Subsection: Category preserving maps}. Thus $%
r^{-1}\left[ A\right] $ is a meager subset of $G$. Therefore, since the
source map $s$ is also open, by Lemma~\ref{Lemma: Kuratowski-Ulam} the set
of $x\in X$ such that $Gx\cap r^{-1}\left[ A\right] $ is meager. This set is
by definition $A^{\bigtriangleup G}$.
\end{proof}

\begin{theorem}
\label{Theorem: Effros}Suppose that $G$ is a Polish groupoid, $X$ is a
Polish $G$-space, and $x\in X$. Denote by $\left[ x\right] $ the orbit of $x$%
. The following statements are equivalent:
\end{theorem}

\begin{enumerate}
\item $\left[ x\right] $ is a $G_{\delta }$ subset of $X$;

\item $\left[ x\right] $ is a Baire space;

\item $\left[ x\right] $ is nonmeager in itself.
\end{enumerate}

\begin{proof}
$\,$By Proposition~\ref{Proposition: action groupoid} we can assume without
loss of generality that $X=G^{0}$ and $G\curvearrowright G^{0}$ is the
standard action. The only nontrivial implication is $3\Rightarrow 1$. After
replacing $G$ with the restriction of $G$ to the closure of $\left[ x\right] 
$, we can assume that $\left[ x\right] $ is dense in $G^{0}$ and hence
nonmeager in $G^{0}$. By Proposition~\ref{Proposition: Borel classes}, $%
\left[ x\right] $ is a Borel subset of $G^{0}$ and in particular it has the
Baire property. Therefore by~\cite[Proposition 8.23]{kechris_classical_1995}
the orbit $\left[ x\right] $ is the union of a meager set $M$ and a $%
G_{\delta }$ set $U$. One can conclude that $\left[ x\right] =U^{\ast G}$
arguing as in~\cite[Proposition 4.4]{sami_polish_1994}. Clearly $\left[ x%
\right] $ is the union of $U^{\ast G}$ and $M^{\bigtriangleup G}$. By Lemma~%
\ref{Lemma: meager transform}, $M^{\bigtriangleup G}$ is meager and hence,
since $\left[ x\right] $ is nonmeager, $M^{\bigtriangleup G}=\varnothing $.
Therefore $\left[ x\right] =U^{\ast G}$ is $G_{\delta }$ by Lemma~\ref%
{Lemma: Borel class Vaught transform}.
\end{proof}

Theorem~2.1 of~\cite{ramsay_mackey-glimm_1990} asserts that it is equivalent
for the conditions in Theorem~\ref{Theorem: Effros} to hold for all points
of $X$.

Suppose that $G$ is a Polish groupoid, $X$ is a Polish $G$-space, and $x\in
G^{0}$. The fiber $Gp(x)$ is a Polish space, and the stabilizer $G_{x}$ of $%
x $ is a Polish group acting from the right by composition on $Gp(x)$. One
can then consider the quotient space $Gp(x)/G_{x}$ and the quotient map $\pi
_{x}:Gp(x)\rightarrow Gp(x)/G_{x}$, which is clearly continuous and open.
When $G\curvearrowright G^{0}$ is the standard action of $G$ on its set of
objects and $x\in G^{0}$, then the stabilizer $G_{x}$ is just $xGx$.

It is not difficult to see that the proof of~\cite[Theorem 3.2]%
{ramsay_mackey-glimm_1990} can be adapted to the context where $G$ is a not
necessarily regular Polish groupoid, as observed in~\cite[page 362]%
{ramsay_mackey-glimm_1990}. The following lemma can then be obtained as an
immediate consequence.

\begin{lemma}[Ramsay]
\label{Lemma: Effros-Ramsay}Suppose that $G$ is a Polish groupoid, and $x\in
G^{0}$. If the orbit $\left[ x\right] $ of $x$ is $G_{\delta }$, then the
map $\phi _{x}:Gx/xGx\rightarrow \left[ x\right] $ defined by $\phi
_{x}\left( \pi (\gamma )\right) =r(\gamma )$ is a homeomorphism.
\end{lemma}

\begin{corollary}
\label{Corollary: Effros-Ramsay}Suppose that $G$ is a Polish groupoid, $X$
is a Polish $G$-space, and $x\in X$. If the orbit $\left[ x\right] $ of $x$
is $G_{\delta }$, then the map $\phi _{x}:Gp(x)/G_{x}\rightarrow \left[ x%
\right] $ defined by $\phi _{x}\left( \pi (\gamma )\right) =\gamma x$ is a
homeomorphism.
\end{corollary}

\begin{proof}
Consider the action groupoid $G\ltimes X$, and let us identify $X$ with the
space of objects of $G\ltimes X$ as in Proposition~\ref{Proposition: action
groupoid}. Consider the map $\psi $ defined by 
\begin{eqnarray*}
Gp(x) &\rightarrow &\left( G\ltimes X\right) x \\
\gamma &\mapsto &(\gamma ,x)\text{.}
\end{eqnarray*}%
Observe that $\psi $ is a continuous map with continuous inverse%
\begin{eqnarray*}
\left( G\ltimes X\right) x &\rightarrow &Gp(x) \\
(\gamma ,x) &\mapsto &\gamma \text{.}
\end{eqnarray*}%
Moreover the image of $G_{x}$ under $\psi $ is precisely $x\left( G\ltimes
X\right) x$. The proof is then concluded by invoking Lemma~\ref{Lemma:
Effros-Ramsay}.
\end{proof}

\subsection{A Polish topology on quotient spaces}

Suppose in this subsection that $G$ is a Polish groupoid, which is moreover
regular. Equivalently the topology of $G$ is (globally) Polish. The
following lemma is proved in~\cite[page 362]{ramsay_mackey-glimm_1990}.

\begin{lemma}[Ramsay]
\label{Lemma: Ramsay}Suppose that $G$ is a regular Polish groupoid. If $U$
is an open subset of $G$ containing the set of objects $G^{0}$, then there
is an open subset $V$ of $G$ containing the set of objects $G^{0}$ such that 
$VV\subset U$.
\end{lemma}

Fix $x\in G^{0}$. If $V$ is a neighborhood of $G^{0}$ in $G$, define the set%
\begin{equation*}
A_{V,x}=\left\{ \left( \rho ,\gamma \right) \in Gx\times Gx:\rho \gamma
^{-1}\in V\right\} \text{.}
\end{equation*}%
Observe that, if $\gamma \in Gx$, then the collection of open subsets of the
form $V\gamma $, where $V$ is an open neighborhood of $r(\gamma )$ in $G$,
is a basis of neighborhoods of $\gamma $ in $Gx$. It follows from this
observation and Lemma~\ref{Lemma: Ramsay} that the collection 
\begin{equation*}
\mathcal{U}_{x}=\left\{ A_{V,x}:V\text{ is a neighborhood of }G^{0}\text{ in 
}G\right\}
\end{equation*}%
generates a uniformity compatible with the topology of $Gx$.

Suppose now that $H$ is a closed subgroup of $xGx$, and consider the right
action of $H$ on $Gx$ by translation. Denote by $\pi $ the quotient map $%
Gx\rightarrow Gx/H$. Observe that $\pi $ is continuous and open. If $V$ is a
neighborhood of $G^{0}$ in $G$ define%
\begin{equation*}
A_{V,x,H}=\left\{ \left( \pi (\gamma ),\pi (\rho )\right) \in Gx/H\times
Gx/H:\rho h\gamma ^{-1}\in V\text{ for some }h\in xGx\right\} \text{.}
\end{equation*}%
As before the collection%
\begin{equation*}
\mathcal{U}_{x,H}=\left\{ A_{V,x,H}:V\text{ is a neighborhood of }G^{0}\text{
in }G\right\}
\end{equation*}%
generates a uniformity compatible with the topology of $Gx/H$.

\begin{proposition}
\label{Proposition: Polish quotient fiber}The quotient $Gx/H$ is a Polish
space.
\end{proposition}

\begin{proof}
The topology on $Gx/H$ is induced by a countably generated uniformity, and
hence it is metrizable. Since the quotient map $\pi :Gx\rightarrow Gx/H$ is
continuous and open, it follows from~\cite[Theorem 2.2.9]{gao_invariant_2009}
that $Gx/H$ is Polish.
\end{proof}

\begin{proposition}
\label{Proposition: Effros Polish}Suppose that $G$ is a regular Polish
groupoid, and $x\in G^{0}$. Denote by $\pi $ the quotient map%
\begin{equation*}
\pi :Gx\rightarrow Gx/xGx\text{.}
\end{equation*}%
The following statements are equivalent:

\begin{enumerate}
\item The orbit $\left[ x\right] $ of $x$ is a $G_{\delta }$ subset of $%
G^{0} $;

\item The map $\phi _{x}:Gx/xGx\rightarrow \left[ x\right] $ defined by $%
\phi _{x}\left( \pi (\gamma )\right) =r(\gamma )$ is a homeomorphism.
\end{enumerate}
\end{proposition}

\begin{proof}
The quotient space $Gx/xGx$ is Polish by Proposition~\ref{Proposition:
Polish quotient fiber}. Therefore if $\phi _{x}$ is a homeomorphism, then $%
\left[ x\right] $ is Polish, and hence a $G_{\delta }$ subset of $G^{0}$ by 
\cite[Theorem 3.11]{kechris_classical_1995}. The converse implication
follows from Lemma~\ref{Lemma: Effros-Ramsay}.
\end{proof}

\subsection{\texorpdfstring{$G_\delta $}{Gdelta} orbits}

\begin{lemma}
\label{Lemma: generic ergodicity}Suppose that $G$ is a Polish groupoid, and $%
\left( U_{n}\right) _{n\in \omega }$ is an enumeration of a basis of
nonempty open subsets of $G^{0}$. If $G$ has a dense orbit, then every
element of $\bigcap_{n}\left[ U_{n}\right] $ has dense orbit.
\end{lemma}

The proof of Lemma~\ref{Lemma: generic ergodicity} is immediate. Recall that 
$\left[ U_{n}\right] $ denotes the $G$-saturation $r\left[ s^{-1}\left[ U_{n}%
\right] \right] $ of $U_{n}$.

\begin{lemma}
\label{Lemma: closure}Suppose that $G$ is a Polish groupoid. Define the
equivalence relation $\overline{E}$ on $G^{0}$ by $(x,y)\in \overline{E}$
iff the orbits of $x$ and $y$ have the same closure. The equivalence
relation $\overline{E}$ is $G_{\delta }$ and contains $E_{G}$.
\end{lemma}

\begin{proof}
Suppose that $\left( U_{n}\right) _{n\in \omega }$ is an enumeration of a
countable open basis of $G^{0}$. We have that $(x,y)\in \overline{E}$ if and
only if $\forall n\in \omega $, $x\in \left[ U_{n}\right] $ iff $y\in \left[
U_{n}\right] $. It follows that $\overline{E}$ is $G_{\delta }$.
\end{proof}

\begin{lemma}
\label{Lemma: T0 quotient}Suppose that $G$ is a Polish groupoid such that
every orbit of $G$ is $G_{\delta }$. If $x,y\in G^{0}$ are such that $\left[
x\right] \neq \left[ y\right] $ and $\left[ y\right] \cap \overline{\left[ x%
\right] }\neq \varnothing $ then $\overline{\left[ y\right] }\cap \left[ x%
\right] =\varnothing $. Equivalently the quotient space $G^{0}\left/
E_{G}\right. $ is $T_{0}$
\end{lemma}

\begin{proof}
After replacing $G$ with the restriction of $G$ to $\overline{\left[ x\right]
}$ we can assume that $\overline{\left[ y\right] }\subset \overline{\left[ x%
\right] }=X$. Denote by $\left( U_{n}\right) _{n\in \omega }$ an enumeration
of a basis of nonempty open subsets of $G^{0}$. By Lemma~\ref{Lemma: generic
ergodicity}, every element of $\bigcap_{n}\left[ U_{n}\right] $ has dense
orbit. Since $\left[ x\right] \cap \left[ y\right] =\varnothing $, $\left[ y%
\right] $ is not dense in $X$ (otherwise it would be comeager and it would
intersect $\left[ x\right] $). It follows that, for some $n\in \omega $, $%
y\notin \left[ U_{n}\right] $ and hence $\left[ y\right] \cap
U_{n}=\varnothing $. This shows that $\overline{\left[ y\right] }\subset
X\backslash U_{n}$. On the other hand $U_{n}$ is invariant dense open and $%
\left[ x\right] $ is comeager, hence $\left[ x\right] \subset U_{n}$. This
shows that $\overline{\left[ y\right] }\cap \left[ x\right] =\varnothing $.
\end{proof}

\begin{lemma}
\label{Lemma: Gdelta smooth}Suppose that $G$ is a Polish groupoid, and $X$
is a Polish $G$-space. If every orbit is $G_{\delta }$, then $E_{G}^{X}$ is
smooth.
\end{lemma}

\begin{proof}
By Proposition~\ref{Proposition: action groupoid} we can assume that $%
X=G^{0} $ and $G\curvearrowright G^{0}$ is the standard action. Observe that
if $x,y\in G^{0}$, then $\left[ x\right] =\left[ y\right] $ if and only if $%
\left[ x\right] $ and $\left[ y\right] $ have the same closure. This shows
that the map $x\mapsto \overline{\left[ x\right] }$ from $G^{0}$ to the
space $F(G^{0})$ of closed subsets of $x$ endowed with the Effros Borel
structure is a reduction from $E_{G}^{X}$ to equality in $G^{0}$. It remains
to observe that such a map is Borel. In fact if $U$ is an open subset of $%
G^{0}$, then%
\begin{equation*}
\left\{ x\in G^{0}:\overline{\left[ x\right] }\cap U\neq \varnothing
\right\} =\left[ U\right] =r\left[ s^{-1}\left[ U\right] \right]
\end{equation*}%
is open.
\end{proof}

\begin{proposition}
\label{Proposition: Gdelta orbits}Suppose that $G$ is a Polish groupoid, and 
$X$ is a Polish $G$-space. The following statements are equivalent:

\begin{enumerate}
\item Every orbit is $G_{\delta }$;

\item The orbit equivalence relation $E_{G}^{X}$ is $G_{\delta }$;

\item The quotient space $X/E_{G}^{X}$ is $T_{0}$.

\item The quotient topology generates the quotient Borel structure
\end{enumerate}
\end{proposition}

\begin{proof}
In view of Proposition~\ref{Proposition: action groupoid} we can assume
without loss of generality that $X=G^{0}$ and $G\curvearrowright G^{0}$ is
the standard action.

\begin{description}
\item[(1)$\Rightarrow $(2)] Consider the equivalence relation $\overline{E}$
defined as in~\ref{Lemma: closure}. Suppose that $x,y\in X$ are such that $%
(x,y)\in \overline{E}$. It follows that $\left[ x\right] $ and $\left[ y%
\right] $ are both dense subsets of $Y=\overline{\left[ x\right] }=\overline{%
\left[ y\right] }$. Since both the orbit of $x$ and $y$ are $G_{\delta }$, $%
\left[ x\right] $ and $\left[ y\right] $ are comeager subsets of $Y$. It
follows that they are not disjoint, and hence $\left[ x\right] =\left[ y%
\right] $. This shows that $E_{G}=\overline{F}$ and in particular $E_{G}$ is 
$G_{\delta }$.

\item[(2)$\Rightarrow $(1)] Obvious.

\item[(2)$\Rightarrow $(3)] Follows from Lemma~\ref{Lemma: T0 quotient}.

\item[(3)$\Rightarrow $(1)] Since the quotient map $\pi :X\rightarrow
X\left/ E_{G}\right. $ is continuous and open, $X/E_{G}$ has a countable
basis $\left\{ U_{n}:n\in \omega \right\} $. If $x\in X$ then 
\begin{equation*}
\left[ x\right] =\bigcap \left\{ \pi ^{-1}\left[ U_{n}\right] :n\in \omega 
\text{, }\pi (x)\in U_{n}\right\} \text{.}
\end{equation*}%
This shows that $\left[ x\right] $ is $G_{\delta }$.

\item[(3)$\Rightarrow $(4)] The Borel structure generated by the quotient
topology is separating and countably generated. By~\cite[Theorem 4.2]%
{mackey_borel_1957} it must coincide with the quotient Borel structure.

\item[(4)$\Rightarrow $(3)] Observe that the orbits are Borel. Therefore the
quotient Borel structure is separating and hence the quotient topology
separates points, i.e.\ it is $T_{0}$.
\end{description}
\end{proof}

The equivalence of the conditions in Proposition~\ref{Proposition: Gdelta
orbits} has been proved in~\cite[Theorem 2.1]{ramsay_mackey-glimm_1990}
under the additional assumption that the orbit equivalence relation is $%
F_{\sigma }$.

\subsection{The Glimm-Effros dichotomy}

Denote by $E_{0}$ the orbit equivalence relation on $2^{\omega }$ defined by 
$(x,y)\in E_{0}$ iff $x(n)=y(n)$ for all but finitely many $n\in \omega $.
Observe that $E_{0}$ can be regarded as the (principal) Polish groupoid
associated with the free action of $\bigoplus_{n\in \omega }\mathbb{Z}/2%
\mathbb{Z}$ on $\prod_{n\in \omega }\mathbb{Z}/2\mathbb{Z}$ by translation.
The proof of the following result is contained in~\cite[Section 4]%
{ramsay_mackey-glimm_1990}. An exposition of the proof in the case of Polish
group actions can be found in~\cite[Theorem 6.2.1]{gao_invariant_2009}.

\begin{proposition}
\label{Proposition: Becker-Kechris-Glimm-Effros}Suppose that $G\ $is a
Polish groupoid. If $E_{G}$ is dense and meager in $G^{0}\times G^{0}$, then 
$E_{0}\sqsubseteq _{c}G$.
\end{proposition}

Recall that $E_{0}\sqsubseteq _{c}G$ means that there is an injective
continuous functor $F:E_{0}\rightarrow G$ such that the restriction of $F$
to the set of objects is a Borel reduction from $E_{0}$ to $E_{G}$; see
Subsection~\ref{Subsection: functorial reducibility}. One can then obtain
the following consequences:

\begin{proposition}
\label{Proposition: no Gdelta}Suppose that $G$ is a Polish groupoid. If $G$
has no $G_{\delta }$ orbits, then $E_{0}\sqsubseteq _{c}G$.
\end{proposition}

\begin{proof}
After replacing $G$ with the restriction of $G$ to a class of the
equivalence relation $\overline{E}$ defined as in Lemma~\ref{Lemma: closure}%
, we can assume that every orbit is dense. By Theorem~\ref{Theorem: Effros}
every orbit is meager. It follows from Lemma~\ref{Lemma: Kuratowski-Ulam}
that $E_{G}$ is meager. One can now apply Proposition~\ref{Proposition:
Becker-Kechris-Glimm-Effros}.
\end{proof}

\begin{theorem}
\label{Theorem: Becker-Kechris-Glimm-Effros}Suppose that $G\ $is a Polish
groupoid. If every $G_{\delta }$ orbit is $F_{\sigma }$, then either $E_{G}$
is $G_{\delta }$ or $E_{0}\sqsubseteq _{c}G$.
\end{theorem}

\begin{proof}
Suppose that $E_{0}\not\sqsubseteq _{c}G$. In particular for every $%
G_{\delta }$ subspace $Y$ of $G^{0}$, $E_{0}\not\sqsubseteq _{c}G_{|Y}$.
Denote by $\overline{E}$ the equivalence relation defined as in Lemma~\ref%
{Lemma: closure}. If $y\in X$ define $Y=\left[ y\right] _{\overline{E}}$ and
observe that $Y$ is an $E_{G}$-invariant $G_{\delta }$ subset of $G^{0}$.
Moreover every $E_{G}$-orbit contained in $Y$ is dense in $Y$. Since $E_{0}$
is not continuously reducible to $G$, by Proposition \ref{Proposition: no
Gdelta} there is $z\in Y$ such that $\left[ z\right] _{E_{G}}$ is a dense $%
G_{\delta }$ subset of $Y$. In particular $\left[ z\right] _{E_{G}}$ is $%
G_{\delta }$ subset of $G^{0}$. Therefore by assumption also $Y\backslash %
\left[ z\right] _{G}$ is $G_{\delta }$. Since every orbit of $Y$ is dense, $%
Y\backslash \left[ z\right] _{E_{G}}$ must be empty and $\left[ z\right]
_{E_{G}}=Y=\left[ y\right] _{\overline{E}}$ is $G_{\delta }$. This shows
that every orbit of $G$ is $G_{\delta }$ and hence $E_{G}$ is $G_{\delta }$
by Proposition~\ref{Proposition: Gdelta orbits}.
\end{proof}

\begin{corollary}
\label{Corollary: functorial E0}Suppose that $G$ and $H$ are Polish
groupoids such that every $G_{\delta }$ orbit is $F_{\sigma }$. If $E_{G}$
and $E_{H}$ are Borel reducible to $E_{0}$, then $G\leq H$ if and only if $%
E_{G}\leq E_{H}$.
\end{corollary}

\begin{proof}
Suppose that $E_{G}\leq E_{H}$. If $E_{H}$ is smooth then the conclusion
follows from Corollary~\ref{Corollary: functorial smooth}. If $E_{G}$ is not
smooth then $G\sim _{B}H\sim _{B}E_{0}$ by Theorem~\ref{Theorem:
Becker-Kechris-Glimm-Effros}; see Definition~\ref{Definition: Borel
bireducibility}.
\end{proof}

We can combine Proposition~\ref{Proposition: Gdelta orbits} with Theorem~\ref%
{Theorem: Becker-Kechris-Glimm-Effros} to get the following result. It was
obtained in~\cite{ramsay_mackey-glimm_1990} under the additional assumption
that the orbit equivalence relation $E_{G}$ is $F_{\sigma }$. The notion of
nonatomic and ergodic Borel measure with respect to an equivalence relation
can be found in~\cite[Definition 6.1.5]{gao_invariant_2009}.

\begin{theorem}
\label{Theorem: GE dichotomy non smooth}Suppose that $G\ $is a Polish
groupoid such that every $G_{\delta }$ orbit is $F_{\sigma }$. The following
statements are equivalent:

\begin{enumerate}
\item There is an orbit which is not $G_{\delta }$;

\item $E_{0}\sqsubseteq _{c}G$;

\item $E_{0}\leq _{B}E_{G}$;

\item There is an $E_{G}$-nonatomic $E_{G}$-ergodic Borel probability
measure on $G^{0}$;

\item $E_{G}$ is not smooth;

\item $E_{G}$ is not $G_{\delta }$;

\item Some orbit is not open in its closure.
\end{enumerate}
\end{theorem}

\begin{proof}
$\,$The implication (1)$\Rightarrow $(2) follows from Theorem~\ref{Theorem:
Becker-Kechris-Glimm-Effros}. The implication (2)$\Rightarrow $(3) is
obvious. For (3)$\Rightarrow $(4), observe that if $f:2^{\omega }\rightarrow
X$ is a Borel reduction from $E_{0}$ to $E_{G}$, $\mu $ is the product
measure on $2^{\omega }$, and $\nu $ is the push-forward of $\mu $ under $f$%
, then $\nu $ is an $E_{G}$-nonatomic and $E_{G}$-ergodic Borel probability
measure on $G^{0}$. The implication (4)$\Rightarrow $(5) follows from~\cite[%
Proposition 6.1.6]{gao_invariant_2009}. By Lemma~\ref{Lemma: Gdelta smooth}
(5) implies 6). The implication (6)$\Rightarrow $(1) is contained in
Proposition~\ref{Proposition: Gdelta orbits}. Since a set that is open in
its closure is $G_{\delta }$, the implication (1)$\Rightarrow $(7) is
obvious. Let us show that (7)$\Rightarrow $(1). Suppose that every orbit is $%
G_{\delta }$, and fix $x\in G^{0}$. After replacing $G$ with its restriction
to the closure of the orbit of $x$, we can assume that $x$ has dense orbit.
Therefore $\left[ x\right] $ is a dense $G_{\delta }$ in $x$. Since $\left[ x%
\right] $ is by assumption also $F_{\sigma }$, $\left[ x\right]
=\bigcup_{n}F_{n}$ where the $F_{n}$'s are closed in $X$. Being $\left[ x%
\right] $ nonmeager in $X$, there is an open subset $U$ of $X$ contained in $%
F_{n}$ for some $n\in \omega $. Hence $\left[ x\right] =\left[ U\right] $ is
open.
\end{proof}

\section{Better topologies\label{Section: better}}

\subsection{Polishability of Borel $G$-spaces\label{Subsection:
Polishability Borel}}

\begin{theorem}
\label{Theorem: Becker-Kechris Borel}Suppose that $G$ is a Polish groupoid.
Every Borel $G$-space is Borel $G$-isomorphic to a Polish $G$-space.
Equivalently if $X$ is a Borel $G$-space, then there is a Polish topology
compatible with the Borel structure of $G$ that makes the action of $G$ on $%
X $ continuous.
\end{theorem}

Theorem~\ref{Theorem: Becker-Kechris Borel} answers a question of Ramsay
from~\cite{ramsay_polish_2000}. The rest of this subsection is dedicated to
the proof of Theorem~\ref{Theorem: Becker-Kechris Borel}. The analogous
statement for actions of Polish groups is proved in a similar way in~\cite[%
Theorem 5.2.1]{becker_descriptive_1996}. Fix a countable basis $\mathcal{A}$
of Polish open subsets of $G$. Suppose that $X$ is a Polish $G$-space. We
want to define a topology $t$ on $X$ such that

\begin{enumerate}
\item $t$ is Polish,

\item the action $G\curvearrowright \left( X,t\right) $ is continuous, i.e.\
the anchor map $p:X\rightarrow G^{0}$ is continuous and 
\begin{eqnarray*}
G\ltimes X &\rightarrow &X \\
(\gamma ,x) &\mapsto &\gamma x
\end{eqnarray*}%
is continuous,

\item $t$ generates the Borel structure of $X$.
\end{enumerate}

By Lemma~\ref{Lemma: Borel Vaught transform} and~\cite[5.1.3 and 5.1.4]%
{becker_descriptive_1996} there exists a countable Boolean algebra $\mathcal{%
B}$ of Borel subsets of $X$ satisfying the following conditions:

\begin{itemize}
\item For all $B\in \mathcal{B}$ and $U\in \mathcal{A}$, $B^{\bigtriangleup
U}\in \mathcal{B}$;

\item The topology $t^{\prime }$ generated by the basis $\mathcal{B}$ is
Polish.
\end{itemize}

Observe that the identity function from $X$ with its original Borel
structure to $\left( X,t^{\prime }\right) $ is Borel measurable, and hence a
Borel isomorphism by~\cite[Theorem 15.1]{kechris_classical_1995}. It follows
that $t^{\prime }$ generates the Borel structure of $X$. Define $\mathcal{S}$
to be the set 
\begin{equation*}
\left\{ B^{\bigtriangleup U}:B\in \mathcal{B}\text{, }U\in \mathcal{A}%
\right\} \text{,}
\end{equation*}%
and $t$ to be the topology on $X$ having $\mathcal{S}$ as subbasis.

\begin{claim-no}
The action $G\curvearrowright \left( X,t\right) $ is continuous.
\end{claim-no}

\begin{proof}
If $V\in \mathcal{A}$ then%
\begin{equation*}
p^{-1}\left[ s\left[ V\right] \right] =X^{\bigtriangleup V}\in \mathcal{S}%
\text{.}
\end{equation*}%
This shows that $p:X\rightarrow G^{0}$ is $t$-continuous. Let us now show
that the map $G\ltimes X\rightarrow X$ is $t$-continuous. Suppose that $B\in 
\mathcal{B}$, $U\in \mathcal{A}$, and $\left( \gamma _{0},x_{0}\right) \in
G\ltimes X$ is such that $\gamma _{0}x_{0}\in B^{\bigtriangleup U}$. By
Lemma~\ref{Lemma: shift Vaught} there are $W,V\in \mathcal{A}$ such that $%
VW^{-1}\subset U$, $x_{0}\in B^{\bigtriangleup V}$, and $\gamma _{0}\in W$.
We claim that $\gamma x\in B^{\bigtriangleup U}$ for every $x\in
B^{\bigtriangleup V}$ and $\gamma \in W$. Fix $x\in B^{\bigtriangleup V}$
and $\gamma \in W$ and observe that $V\gamma ^{-1}\subset VW^{-1}\subset U$
and hence it follows from Lemma~\ref{Lemma: shift Vaught} that $\gamma x\in
B^{\bigtriangleup U}$. This shows that the action is continuous.
\end{proof}

\begin{claim-no}
The space $\left( X,t\right) $ is $T_{1}$.
\end{claim-no}

\begin{proof}
Pick distinct points $x,y$ of $X$. If $p(x)\neq p(y)$ then there are
disjoint $V,W\in \mathcal{A}$ such that $p(x)\in V$ and $p(y)\in W$. Thus $%
p^{-1}\left[ V\right] $ and $p^{-1}\left[ W\right] $ are open sets
separating $x$ and $y$. Suppose that $p(x)=p(y)$. Consider the function $%
f:Gp(x)\rightarrow X\times X$ defined by%
\begin{equation*}
f(\gamma )=\left( \gamma x,\gamma y\right) \text{.}
\end{equation*}%
Observe that $f$ is Borel when $X\times X$ is endowed with the $t^{\prime
}\times t^{\prime }$ topology. By~\cite[Theorem 8.38]{kechris_classical_1995}
there is a dense $G_{\delta }$ subset $Q$ of $Gx$ such that the restriction
of $f$ to $Q$ is $\left( t^{\prime }\times t^{\prime }\right) $-continuous.
Let $\gamma _{0}\in Q$. Since $\mathcal{B}$ is a basis for the Polish
topology $t^{\prime }$ on $X$ there are disjoint elements $B,C$ of $\mathcal{%
B}$ such that $\gamma _{0}x\in B$ and $\gamma _{0}y\in C$. Since $f$ is $%
\left( t^{\prime }\times t^{\prime }\right) $-continuous on $Q$ there is $%
U\in \mathcal{A}$ such that $Up(x)\neq \varnothing $ and%
\begin{equation*}
f\left[ Up(x)\cap Q\right] \subset B\times C\text{.}
\end{equation*}%
Thus $\forall ^{\ast }\gamma \in Up(x)$, $\gamma x\in B$ and $\gamma y\in C$%
. This shows that%
\begin{equation*}
x\in B^{\bigtriangleup U}\text{, }y\in C^{\bigtriangleup U}\text{, }y\notin
B^{\bigtriangleup U}\text{, and }x\notin C^{\bigtriangleup U}\text{.}
\end{equation*}%
This concludes the proof that $\left( X,t\right) $ is T$_{1}$.
\end{proof}

\begin{claim-no}
The space $\left( X,t\right) $ is regular.
\end{claim-no}

\begin{proof}
Suppose that $B\in \mathcal{B}$ and $U\in \mathcal{A}$. Pick $x_{0}\in
B^{\bigtriangleup U}$. It is enough to show that there is a $t$-open subset $%
N$ of $B^{\bigtriangleup U}$ containing $x_{0}$ such that the $t$-closure of 
$N$ is contained in $B^{\bigtriangleup V}$. Since $x_{0}\in
B^{\bigtriangleup U}$ by Lemma~\ref{Lemma: shift Vaught} there are $%
W_{1},V_{1}\in \mathcal{A}$ such that $V_{1}W_{1}^{-1}\cup V_{1}\subset U$, $%
p(x_{0})\in W_{1}$, and $x_{0}\in B^{\bigtriangleup V_{1}}$. Since $x_{0}\in
B^{\bigtriangleup V_{1}}$ again by Lemma~\ref{Lemma: shift Vaught} there are 
$V_{2},W_{2}\in \mathcal{A}$ such that $V_{2}W_{2}^{-1}\subset V_{1}$, $%
p(x_{0})\in W_{2}$, and $x_{0}\in B^{\bigtriangleup V_{2}}$. Define $W\in 
\mathcal{A}$ such that 
\begin{equation*}
p(x_{0})\in W\subset W_{1}^{-1}\cap W_{2}.
\end{equation*}%
Consider 
\begin{equation*}
N=B^{\bigtriangleup V_{2}}\cap p^{-1}s\left[ W\right]
\end{equation*}%
and observe that $N$ is a $t$-open subset of $X$ containing $x_{0}$. We
claim that the closure of $N$ is contained in $B^{\bigtriangleup U}$. Define 
$F=\left( B^{\bigtriangleup V_{2}}\right) ^{\ast W}$ and observe that $F$ is
relatively closed in $p^{-1}s\left[ W\right] $ by Lemma~\ref{Lemma: basic
Vaught}(4). We claim that 
\begin{equation*}
N\subset F\subset B^{\bigtriangleup U}.
\end{equation*}%
Suppose that $x\in N$. If $\gamma \in Wp(x)$ we have that%
\begin{equation*}
V_{2}\gamma ^{-1}\subset V_{2}W^{-1}\subset V_{2}W_{2}^{-1}\subset V_{1}%
\text{.}
\end{equation*}%
Therefore $\gamma x\in B^{\bigtriangleup V_{1}}$. Being this true for every $%
\gamma \in Wp(x)$, $x\in \left( B^{\bigtriangleup V_{1}}\right) ^{\ast W}=F$%
. Suppose now that $x\in F$ and pick $\gamma \in Wp(x)$ such that $\gamma
x\in B^{\bigtriangleup V_{1}}$. We thus have%
\begin{equation*}
V_{1}\gamma \subset V_{1}W\subset V_{1}W_{1}^{-1}\subset U
\end{equation*}%
which implies by Lemma~\ref{Lemma: shift Vaught} that $x=\gamma ^{-1}(\gamma
x)\in B^{\bigtriangleup U}$. This concludes the proof that $N\subset
F\subset B^{\bigtriangleup U}$. We will now show that the $t$-closure of $N$
is contained in $B^{\bigtriangleup U}$. It is enough to show that if $%
x\notin B^{\bigtriangleup U}$ then there is a $t$-open neighborhood of $x$
disjoint from $N$. This is clear if $p(x)\notin s\left[ W\right] $. Suppose
now that $p(x)\in s\left[ W\right] $. Since $F$ is relatively closed in $%
p^{-1}s\left[ W\right] $ and%
\begin{equation*}
N\subset F\subset B^{\bigtriangleup V}\cap p^{-1}s\left[ W\right]
\end{equation*}%
we have that $p^{-1}s\left[ W\right] \backslash F$ is an open subset of $X$
containing $x$ and disjoint from $N$. This concludes the proof that the
closure of $N$ is contained in $B^{\bigtriangleup V}$. We have thus found an
open neighborhood $N$ of $x$ whose closure is contained in $%
B^{\bigtriangleup V}$. This concludes the proof that $\left( X,t\right) $ is
regular.
\end{proof}

\begin{claim-no}
The space $\left( X,t\right) $ is strong Choquet.
\end{claim-no}

\begin{proof}
Define $\mathcal{C}$ to be the (countable) set of nonempty finite
intersections of elements of $\mathcal{S}$ and observe that $\mathcal{C}$ is
a basis for $\left( X,t\right) $. Fix a well ordering $E$ of the countable
set $\mathcal{C\times B}\times \mathcal{A}$. Let $d^{\prime }$ be a complete
metric on $X$ compatible with the Polish topology $t^{\prime }$. We want to
define a strategy for Player II in the strong Choquet game; see~\cite[%
Section 8.D]{kechris_classical_1995}. Suppose that Player I plays $t$-open
sets $N_{i}$ for $i\in \omega $ and $x_{i}\in N_{i}$. At the $i$-th turn
Player II will choose an element $\left( M_{i},B_{i},U_{i}\right) $ of $%
\mathcal{C}\times \mathcal{B}\times \mathcal{A}$ in such a way that the
following properties hold:

\begin{enumerate}
\item $x_{i}\in M_{i}$;

\item The $t$-closure of $M_{i}$ is contained in $N_{i}$;

\item The closure of $U_{i+1}$ in $U_{0}$ is contained in $U_{i}$;

\item The $t^{\prime }$-closure of $B_{i+1}$ is contained in $B_{i}$;

\item The $d^{\prime }$-diameter of $B_{i}$ is less than $2^{-i}$;

\item The $d_{U_{0}}$-diameter of $U_{i}$ is less than $2^{-i}$ for $i\geq 1$%
, where $d_{U_{0}}$ is a compatible complete metric on $U_{0}$;

\item $M_{i}\subset B_{i}^{\bigtriangleup U_{i}}$.
\end{enumerate}

Player II strategy is the following: At the $i$-th turn pick the $E$-least
tuple $\left( M_{i},B_{i},U_{i}\right) $ in $\mathcal{C}\times \mathcal{B}%
\times \mathcal{A}$ satisfying properties (1)--(7). We need to show that the
set of such tuples is nonempty. Observe that $x_{i}\in N_{i}\subset
M_{i-1}\subset B_{i-1}^{\bigtriangleup U_{i-1}}$. Thus $U_{i-1}p(x_{i})\neq
\varnothing $ and $\exists ^{\ast }\gamma \in U_{i-1}p(x_{i})$ such that $%
\gamma x_{i}\in B_{i-1}$. Since $\mathcal{B}$ is a basis for $\left(
X,t^{\prime }\right) $ and $\mathcal{A}$ is a basis for $G$ we can find $%
B_{i}$ and $U_{i}$ such that

\begin{itemize}
\item (3)--(6) hold.

\item $U_{i}p(x_{i})\neq \varnothing $, and

\item $\exists ^{\ast }\gamma \in U_{i}p(x_{i})$ such that $\gamma x_{i}\in
B_{i}$.
\end{itemize}

Consider $M=B_{i}^{\bigtriangleup U_{i}}\cap N_{i}$ and observe that $M$ is
a $t$-open set containing $x_{i}$. Since $\left( X,t\right) $ is regular
there is $M_{i}\in \mathcal{C}$ such that $x_{i}\in M_{i}$ and the closure
of $M_{i}$ is contained in $N_{i}\cap B_{i}^{\bigtriangleup U_{i}}$. This
ensures that (1),(2),(7) are satisfied. We now show that this gives a
winning strategy for Player II. For every $i\in \omega $ we have that $%
x_{i}\in M_{i}\subset B_{i}^{\bigtriangleup U_{i}}$ and hence there is $%
\gamma _{i}\in U_{i}p(x_{i})$ such that $\gamma _{i}x_{i}=y_{i}\in B_{i}$.
Define $\gamma $ to be the limit of the sequence $(\gamma _{i})_{i\in \omega
}$ in $U_{0}$ and $y$ to be the $t^{\prime }$-limit of the sequence $%
(y_{i})_{i\in \omega }$ in $Y$. Observe that%
\begin{equation*}
p(y)=\lim_{i}p(y_{i})=\lim_{i}r(\gamma _{i})=r(\gamma )\text{.}
\end{equation*}%
Define $x=\gamma ^{-1}y\in X$ and observe that $x$ is the $t$-limit of the
sequence $(x_{i})_{i\in \omega }$. Fix $i\in \omega $. For $j>i$ we have
that $x_{j}\in N_{j}\subset M_{i}$ and hence $x$ is contained in the $t$%
-closure of $M_{i}$, which is in turn contained in $N_{i}$. This shows that $%
x\in \bigcap_{i\in \omega }N_{i}$, concluding the proof that Player II has a
winning strategy in the strong Choquet game in $\left( X,t\right) $.
\end{proof}

The proof of Theorem \ref{Theorem: Becker-Kechris Borel} is finished
recalling that a regular T$_{1}$ strong Choquet space is Polish~\cite[%
Theorem 8.18]{kechris_classical_1995}.

\subsection{Finer topologies for Polish $G$-spaces\label{Subsection: finer
topology}}

\begin{theorem}
\label{Theorem: Becker-Kechris finer}Suppose that $G$ is a Polish groupoid,
and $\left( X,\tau \right) $ is a Polish $G$-space. Assume that $V\subset G$
is an open Polish subset, $P\subset X$ is $\mathbf{\Sigma }_{\alpha }^{0}$
for some $\alpha \in \omega _{1}$, and $Q=P^{\bigtriangleup V}$.There is a
topology $t$ on $X$ such that:

\begin{enumerate}
\item $t$ is a Polish topology;

\item $t$ is finer that $\tau $;

\item $Q$ is $t$-open,

\item The action of $G\ $on $\left( X,t\right) $ is continuous;

\item $t$ has a countable basis $\mathcal{B}$ such that for every $B\in 
\mathcal{B}$ there is $n\in \omega $ such that $B$ is $\mathbf{\Sigma }%
_{\alpha +n}^{0}$ with respect to $\tau $.
\end{enumerate}
\end{theorem}

The analogous statement for actions of Polish groups is proved in a similar
way in~\cite[Theorem 5.1.8]{becker_descriptive_1996}. Let $\mathcal{A}$ be a
countable basis of Polish open subsets for $G$ containing $V$ and let $%
\mathcal{D}$ be a countable basis for $\left( X,\tau \right) $. By Lemma~\ref%
{Lemma: Borel class Vaught transform} and \cite[5.1.3, 5.1.4]%
{becker_descriptive_1996} there is a countable Boolean algebra $\mathcal{B}$
of subsets of $X$ satisfying the following:

\begin{enumerate}
\item For $B\in \mathcal{B}$ and $U\in \mathcal{A}$, $B^{\bigtriangleup
V}\in \mathcal{A}$;

\item $P\in \mathcal{B}$;

\item $\mathcal{D}\subset \mathcal{B}$;

\item $\mathcal{B}$ is a basis for a Polish topology $t^{\prime }$;

\item For every $B\in \mathcal{B}$, there is $n\in \omega $ such that $B$ is 
$\mathbf{\Sigma }_{\alpha +n}$ with respect to $\tau $.
\end{enumerate}

Define 
\begin{equation*}
\mathcal{S}=\left\{ B^{\bigtriangleup V}:B\in \mathcal{B},V\in \mathcal{A}%
\right\} \text{,}
\end{equation*}%
and%
\begin{equation*}
\mathcal{S}^{\ast }=\mathcal{S}\cup \mathcal{D}\text{.}
\end{equation*}%
Consider the topology $t$ on $X$ having $\mathcal{S}^{\ast }$ as a subbasis.
We claim that $t$ is a Polish topology finer that $\tau $ and coarser that $%
t^{\prime }$ making the action continuous. Clearly $t$ is finer that $\tau $
and in particular $p:\left( X,t\right) \rightarrow G^{0}$ is continuous. The
proof that the action is continuous and that $t$ is a Polish topology is
analogous to the proof of Theorem~\ref{Theorem: Becker-Kechris Borel}. The
following corollary can be obtained from Theorem~\ref{Theorem:
Becker-Kechris finer} together with~\cite[Subsection 5.1.3]%
{becker_descriptive_1996}.

\begin{corollary}
Suppose that $G$ is a Polish groupoid, and $\left( X,\tau \right) $ is a
Polish $G$-space. If $\mathcal{J}$ is a countable collection of $G$%
-invariant Borel subsets of $X$, then there is a Polish topology $t$ on $X$
finer than $\tau $ and making the action continuous such that all elements
of $\mathcal{J}$ are $t$-clopen.
\end{corollary}

\section{Borel orbit equivalence relations\label{Section: Borel orbit}}

\subsection{A Borel selector for cosets\label{Subsection: Borel selector
cosets}}

Suppose that $G$ is an open Polish groupoid. Denote by $F(G)$ the (standard)
Borel space of closed subsets of $G$ endowed with the Effros Borel
structure; see Appendix~\ref{Appendix}. A similar proof as~\cite[Theorem
12.13]{kechris_classical_1995} shows that there is a Borel function%
\begin{equation*}
\sigma :F(G)\backslash \left\{ \varnothing \right\} \rightarrow G
\end{equation*}%
such that $\sigma (A)\in A$ for every nonempty closed subset $A$ of $G$.
Denote by $S(G)$ the Borel space of closed subgroupoids of $G$. This is the
Borel subset of $F(G)$ containing the closed subsets $H$ of $G$ such that
for $\gamma ,\rho \in H$, $\gamma ^{-1}\in H$ and if $r(\gamma )=s(\rho )$
then $\rho \gamma \in H$. If $H\in S(G)$ denote by $\sim _{H}$ the
equivalence relation on $G$ defined by $\gamma \sim _{H}\rho $ iff $\gamma
=\rho h$ for some $h\in H$ or, equivalently, $\gamma H=\rho H$.

\begin{proposition}
\label{Proposition: coset selector}The relation $\sim $ on $G\times S(G)$
defined by $\left( \gamma ,H\right) \sim \left( \gamma ^{\prime },H^{\prime
}\right) $ iff $H=H^{\prime }$ and $\gamma H=\gamma ^{\prime }H^{\prime }$
has a Borel transversal $T.$
\end{proposition}

\begin{proof}
Define the map $f$ from $S(G)\times G$ to $F(G)$ by $f\left( \gamma
,H\right) =\gamma H$. We claim that $f$ is Borel. Let us show that if $U$ is
an open subset of $G$ then the set 
\begin{equation*}
A_{U}=\left\{ \left( \gamma ,H\right) \in G\times S(G):\gamma H\cap U\neq
\varnothing \right\}
\end{equation*}%
is Borel. Since the set%
\begin{equation*}
\left\{ \left( \rho ,\gamma ,H\right) \in G\times G\times S(G):\gamma
^{-1}\rho \in H\text{ and }\rho \in U\right\}
\end{equation*}%
is Borel, its projection $A_{U}$ on the last two coordinates is analytic. We
want to show that $A_{U}$ is co-analytic. Fix a countable basis of Polish
open sets $\left\{ U_{n}:n\in \omega \right\} $ for $G$. Observe that $%
\left( H,\gamma \right) \in A_{U}$ if and only if there is $n\in \omega $
such that $\gamma U_{n}\subset U$ and $U_{n}\cap H\neq \varnothing $. It is
now enough to show that $\left\{ \gamma \in G:\gamma U_{n}\subset U\right\} $
is co-analytic. This follows from the fact that%
\begin{equation*}
\left\{ \left( \gamma ,\rho \right) \in G\times U_{n}:\text{either }r(\rho
)\neq s(\gamma )\text{ or }r(\rho )=s(\gamma )\text{ and }\gamma \rho \in
U\right\}
\end{equation*}%
is a Borel set and it co-projection on the first coordinate is $\left\{
\gamma \in G:U_{n}\gamma \subset U\right\} $. If now $\sigma :F(G)\backslash
\left\{ \varnothing \right\} \rightarrow G$ is a Borel map such that $\sigma
(A)\in A$ for every nonempty closed subset $A$ of $G$, define $g\left(
\gamma ,H\right) =\left( \left( \sigma \circ f\right) \left( \gamma
,H\right) ,H\right) $. Observe that $g$ is a Borel selector for $\sim $.
Therefore the set%
\begin{equation*}
T=\left\{ \left( \gamma ,H\right) :g\left( \gamma ,H\right) =\left( \gamma
,H\right) \right\}
\end{equation*}%
is a Borel transversal for $\sim $.
\end{proof}

\begin{corollary}
\label{Corollary: Borel orbits}If $G$ is a Polish groupoid, and $X$ is a
Borel $G$-space, then the orbits are Borel subsets of $X$.
\end{corollary}

\begin{proof}
Observe that the stabilizer $G_{x}$ is a closed subgroup of $p(x)Gp(x)$ by 
\cite[Theorem 9.17]{kechris_classical_1995}. Consider a Borel transversal $%
T_{x}$ for the equivalence relation $\sim _{G_{x}}$. The function $\gamma
\mapsto \gamma x$ from $T\cap Gx$ to $X$ is a 1:1 Borel function from $T_{x}$
onto the orbit of $x$. This shows that the orbit of $x$ is Borel by~\cite[%
Theorem 15.1]{kechris_classical_1995}.
\end{proof}

\subsection{Borel orbit equivalence relations\label{Subsection: Borel orbit
equivalence relations}}

Suppose that $G$ is a Polish groupoid, and $X$ is a Polish $G$-space. If $%
x\in X$ then Lemma~\ref{Lemma: Borel Vaught transform} and \cite[5.1.3 and
5.1.4]{becker_descriptive_1996} show that there is a sequence $\left(
B_{x,n}\right) _{n\in \omega }$ of Borel subsets of $X$ such that $\left[ x%
\right] =B_{x,0}$ and 
\begin{equation*}
\mathcal{B}(x)=\left\{ B_{x,n}:n\in \omega \right\}
\end{equation*}%
is a Boolean algebra that is a basis for a topology $t(x)$ on $X$ making the
action continuous, and such that $B^{\bigtriangleup U}\in \mathcal{B}(x)$
whenever $B\in \mathcal{B}(x)$ and $U\in \mathcal{A}$. It is implicit in the
proof of Lemma~\ref{Lemma: Borel Vaught transform} and~\cite[5.1.3, 5.1.4]%
{becker_descriptive_1996} that, under the additional assumption that the
orbit equivalence relation $E_{G}^{X}$ is Borel, the dependence of the
sequence $\left( \mathcal{B}_{x,n}\right) _{n\in \omega }$ from $x$ is
Borel, i.e.\ the relation%
\begin{equation*}
\mathcal{B}(y,x,n)\Leftrightarrow y\in \mathcal{B}_{x,n}
\end{equation*}%
is Borel. This concludes the proof of the following lemma; see also~\cite[%
Lemma 7.1.3]{becker_descriptive_1996}.

\begin{lemma}
\label{Lemma:sigma-algebra}Suppose that $G$ is a Polish groupoid, and $X$ is
a Polish $G$-space. Assume that $\mathcal{A}$ is a countable basis of Polish
open subsets of $G$. If the orbit equivalence relation $E_{G}^{X}$ is Borel,
then there is a Borel subset $\mathcal{B}$ of $X\times X\times \omega $ such
that, letting 
\begin{equation*}
B_{x,n}=\left\{ y:(y,x,n)\in \mathcal{B}\right\}
\end{equation*}%
and 
\begin{equation*}
\mathcal{B}(x)=\left\{ B_{x,n}:n\in \omega \right\} \text{,}
\end{equation*}%
for every $x\in X$ the following hold:

\begin{enumerate}
\item $\left[ x\right] =B_{x,0}$;

\item $B^{\bigtriangleup U}\in \mathcal{B}(x)$ for every $B\in \mathcal{B}%
(x) $, and $U\in \mathcal{A}$;

\item $\mathcal{B}(x)$ is a Boolean algebra;

\item $\mathcal{B}(x)$ is a basis for a Polish topology $t(x)$ on $X$ making 
$X$ a Polish $G$-space.
\end{enumerate}
\end{lemma}

The following result provides a characterization of the Borel $G$-spaces
with Borel orbit equivalence relation. The analogous result for Polish group
actions is \cite[Theorem 7.1.2]{becker_descriptive_1996}.

\begin{theorem}
\label{Theorem: Borel orbit equivalence relations}Suppose that $G$ is a
Polish groupoid, and $X$ is a Borel $G$-space. The following statements are
equivalent

\begin{enumerate}
\item The function 
\begin{eqnarray*}
X &\rightarrow &F(G) \\
x &\mapsto &G_{x}
\end{eqnarray*}%
is Borel;

\item The function 
\begin{eqnarray*}
X\times X &\rightarrow &F(G) \\
(x,y) &\mapsto &G_{x,y}
\end{eqnarray*}%
is Borel;

\item The orbit equivalence relation $E_{G}^{X}$ is Borel.
\end{enumerate}
\end{theorem}

Recall that, for $x,y\in G^{0}$, $G_{x}$ denotes the (closed) stabilizer%
\begin{equation*}
\left\{ \gamma \in Gp(x):\gamma x=x\right\}
\end{equation*}%
and while $G_{x,y}$ is the set%
\begin{equation*}
\left\{ \gamma \in Gp(x):\gamma x=y\right\} \text{;}
\end{equation*}%
see Subsection~\ref{Subsection: Polish groupoids and groupoid actions}.

\begin{proof}
By Theorem~\ref{Theorem: Becker-Kechris Borel} we can assume without loss of
generality that $X$ is a Polish $G$-space. Fix a countable basis $\mathcal{A}%
=\left\{ U_{n}:n\in \omega \right\} $ of nonempty Polish open subsets of $G$%
. Denote also by $T\subset G\times S(G)$ a Borel transversal for the
relation $\left( \gamma ,H\right) \sim \left( \gamma ^{\prime },H^{\prime
}\right) $ iff $H=H^{\prime }$ and $\gamma H=\gamma H^{\prime }$ as in
Proposition~\ref{Proposition: coset selector}.

\begin{description}
\item[(1)$\Rightarrow $(2)] Fix a nonempty open subset $U$ of $G$. It is
enough to show that the set%
\begin{equation*}
\left\{ (x,y)\in X\times X:U\cap G_{x,y}\neq \varnothing \right\}
\end{equation*}%
is co-analytic. Observe that $G_{x,y}\cap U\neq \varnothing $ if and only if
there is a unique $\gamma \in G$ such that $s(\gamma )=p(x)$, $\left( \gamma
,G_{x}\right) \in T$, and $\gamma G_{x}\cap U\neq \varnothing $. Moreover $%
\gamma G_{x}\cap U\neq \varnothing $ if and only if there is $n\in \omega $
such that $\gamma U_{n}\subset U$ and $U_{n}\cap G_{x}\neq \varnothing $.
Fix $n\in \omega $ and recall that by the proof of Proposition~\ref%
{Proposition: coset selector} $\left\{ \gamma \in G:\gamma U_{n}\subset
U\right\} $ is co-analytic. This concludes the proof that 
\begin{equation*}
\left\{ (x,y)\in X\times X:U\cap G_{x,y}\neq \varnothing \right\}
\end{equation*}%
is co-analytic.

\item[(2)$\Rightarrow $(1)] Obvious.

\item[(1)$\Rightarrow $(3)] Observe that $(x,y)\in E_{G}$ if and only if
there is a unique $\gamma \in T$ such that $\left( \gamma ,G_{x}\right) \in
T $ and $r(\gamma )=y$.

\item[(3)$\Rightarrow $(1)] Suppose that $\mathcal{B}$, $\mathcal{B}(x)$, $%
t(x)$, and $B_{x,n}$ for $x\in X$ and $n\in \omega $ are defined as in Lemma~%
\ref{Lemma:sigma-algebra}. Observe that the orbit $\left[ x\right] =B_{x,0}$
is open in $t(x)$. It follows from Lemma~\ref{Lemma: Effros-Ramsay} that the
map 
\begin{eqnarray*}
Gp(x)\left/ G_{x}\right. &\rightarrow &\left[ x\right] \\
\gamma G_{x} &\mapsto &\gamma x
\end{eqnarray*}%
is a $t(x)$-homeomorphism. We want to show that for every $U\in \mathcal{A}$
the set%
\begin{equation*}
\left\{ x\in X:G_{x}\cap U\neq \varnothing \right\}
\end{equation*}%
is Borel. It is enough to show that for every basic nonempty $U,V$ the set%
\begin{equation*}
\left\{ x\in X:UG_{x}\cap V\neq \varnothing \right\}
\end{equation*}%
is co-analytic. We claim that $UG_{x}\cap V\neq \varnothing $ iff $\exists
U_{m}\subset U$ such that $\forall B\in \mathcal{B}(x)$, $x\in
B^{\bigtriangleup U_{m}}$ implies $x\in B^{\bigtriangleup V}$. In fact
suppose that $UG_{x}\cap V\neq \varnothing $ and pick $m$ such that $%
U_{m}\subset U$ and $U_{m}\gamma \subset V$ for some $\gamma \in G_{x}$. If $%
x\in B^{\bigtriangleup U_{m}}$ then $x=\gamma ^{-1}x\in B^{\bigtriangleup
U_{m}\gamma }$ and hence $x\in B^{\bigtriangleup V}$ by Lemma~\ref{Lemma:
shift Vaught}.\ Conversely suppose that $UG_{x}\cap V=\varnothing $ and hence%
\begin{equation*}
\left\{ \gamma x:\gamma \in U\right\} \cap \left\{ \gamma x:\gamma \in
V\right\} =\varnothing \text{.}
\end{equation*}%
Fix $m\in \omega $. Since the map%
\begin{eqnarray*}
Gp(x)\left/ G_{x}\right. &\rightarrow &\left[ x\right] \\
\gamma G_{x} &\mapsto &\gamma x
\end{eqnarray*}%
is a $t(x)$-homeomorphism, the set $\left\{ \gamma x:\gamma \in
U_{m}\right\} $ is open in $[x]$. Thus there is $B\in \mathcal{B}(x)$ such
that%
\begin{equation*}
x\in B\subset \left\{ \gamma x:\gamma \in U_{m}\right\} \text{.}
\end{equation*}%
Moreover 
\begin{equation*}
\left\{ \gamma \in Gp(x):\gamma x\in B\right\}
\end{equation*}%
is an open subset of $Gp(x)$. Therefore there is $k\in \omega $ such that $%
U_{k}\subset U_{m}$ and%
\begin{equation*}
U_{k}p(x)\subset \left\{ \gamma \in Gp(x):\gamma x\in B\right\} \text{.}
\end{equation*}%
In particular $x\in B^{\bigtriangleup U_{k}}$ but $x\notin B^{\bigtriangleup
V}$.
\end{description}
\end{proof}

\section{Universal actions\label{Section: universal}}

Suppose that $G$ is a Polish groupoid. The space $G$ is fibred over the
space of objects $G^{0}$ via the source map $r:G\rightarrow G^{0}$. One can
then consider the corresponding Effros fibred space $F(G,G^{0})$ of closed
subsets of $G$ contained in $Gx$ for some $x\in G^{0}$; see Subsection~\ref%
{Subsection: Effros fibred space}. Recall that $F(G,G^{0})$ is a standard
Borel space fibred over $G^{0}$ via the Borel map assigning $x$ to a closed
nonempty subset $F$ of $xG$. Moreover $F(G,G^{0})$ has naturally the
structure of Borel $G$-space given by the map%
\begin{equation*}
\left( \gamma ,F\right) \mapsto \gamma F
\end{equation*}%
for $F\subset s(\gamma )G$, where%
\begin{equation*}
\gamma F=\left\{ \gamma \rho :\rho \in F\right\} \text{.}
\end{equation*}%
Similarly the fibred product%
\begin{equation*}
\foo_{n\in \omega }F(G,G^{0})=\left\{ (F_{n})_{n\in \omega }\in F\left(
G,G\right) ^{\omega }:\exists x\in G^{0}\text{ }\forall n\in \omega \text{, }%
F_{n}\subset xG\right\}
\end{equation*}%
is naturally a Borel $G$-space with respect to the coordinate-wise action of 
$G$%
\begin{equation*}
\gamma (F_{n})_{n\in \omega }=\left( \gamma F_{n}\right) _{n\in \omega }%
\text{.}
\end{equation*}

We want to show that the Borel $G$-space $\foo_{n\in \omega }F(G,G^{0})$ is
a universal Borel $G$-space. This means that if $X$ is any Borel $G$-space,
then there is a Borel $G$-embedding $\varphi :X\rightarrow \foo_{n\in \omega
}F(G,G^{0})$; see Subsection~\ref{Subsection: Polish groupoids and groupoid
actions}.

The following lemma is well known. A proof is included for convenience of
the reader.

\begin{lemma}
If $X$ is a Polish space, $A\subset X$, and $E(A)$ is the set of $x\in X$
such that for every neighborhood $V$ of $x$, $V\cap A$ is not meager, then $%
E(A)$ is closed in $X$. Moreover $A$ has the Baire property iff $%
A\bigtriangleup E(A)$ is meager.
\end{lemma}

\begin{proof}
Clearly $E(A)$ is closed, and if $A\bigtriangleup E(A)$ is meager then $A$
has the Baire property. Observe that if $A,B\subset X$ are such that $%
A\bigtriangleup B$ is meager, then $E(A)=E\left( B\right) $. If $A$ has the
Baire property, there is an open subset $U$ of $X$ such that $%
A\bigtriangleup U$ is meager. Thus $E(A)=E\left( U\right) $ is equal to the
closure $\overline{U}$ of $U$. It follows that%
\begin{equation*}
A\bigtriangleup E(A)\subset \left( A\bigtriangleup U\right) \cup \left( 
\overline{U}\backslash U\right)
\end{equation*}%
is meager.
\end{proof}

Suppose that $X$ is a Borel $G$-space. In view of Theorem~\ref{Theorem:
Becker-Kechris Borel} we can assume without loss of generality that $X$ is
in fact a Polish $G$-space. Fix a countable open basis $\mathcal{B}=\left\{
B_{n}:n\in \omega \right\} $ of nonempty open subsets of $X$. Assume further
that $\mathcal{A}$ is a countable basis of Polish open subsets of $G$.
Define for $n\in \omega $ the fibred Borel map $\varphi _{n}:X\rightarrow
F(G,G^{0})$ by setting 
\begin{equation*}
\varphi _{n}(x)=\left( E\left( \left\{ \gamma \in Gp(x):\gamma x\in
B_{n}\right\} \right) \right) ^{-1}.
\end{equation*}
Define the Borel fibred map $\varphi :X\rightarrow \foo_{n\in \omega
}F(G,G^{0})$ by $\varphi (x)=\left( \varphi _{n}(x)\right) _{n\in \omega }$.

\begin{claim-no}
$\varphi $ is Borel measurable
\end{claim-no}

It is enough to show that $\varphi _{n}$ is Borel measurable for every $n\in
\omega $. Suppose that $V\in \mathcal{A}$. We want to show that the set of $%
x\in X$ such that%
\begin{equation*}
E\left( \left\{ \gamma \in Gp(x):\gamma x\in B_{n}\right\} \right) \cap
V\neq \varnothing
\end{equation*}%
is Borel. Observe that 
\begin{equation*}
E\left( \left\{ \gamma \in Gp(x):\gamma x\in B_{n}\right\} \right) \cap
V\neq \varnothing
\end{equation*}%
if and only if $\exists W\in \mathcal{A}$ such that $W\subset V$ and 
\begin{equation*}
\left\{ \gamma \in Gp(x):\gamma x\in B_{n}\right\}
\end{equation*}%
is comeager in $Wp(x)$. The set of such elements $x$ of $X$ is Borel by
Lemma~\ref{Lemma: Montgomery-Novikov}.

\begin{claim-no}
$\varphi $ is $G$-equivariant, i.e.\ $\varphi (\gamma x)=\gamma \varphi (x)$
for $(\gamma ,x)\in G\ltimes X$
\end{claim-no}

It is enough to show that $\varphi _{n}(\gamma x)=\gamma \varphi _{n}(x)$
for $n\in \omega $. Observe that%
\begin{equation*}
\varphi _{n}(\gamma x)=\left( E\left( \left\{ \rho \in Gr(\gamma ):\rho
\gamma x\in B_{n}\right\} \right) \right) ^{-1}
\end{equation*}%
and%
\begin{equation*}
\varphi _{n}(x)=\left( E\left( \left\{ \tau \in Gp(x):\tau x\in
B_{n}\right\} \right) \right) ^{-1}
\end{equation*}%
We thus have to prove that%
\begin{equation*}
\left( E\left( \left\{ \rho \in Gr(\gamma ):\rho \gamma x\in B_{n}\right\}
\right) \right) ^{-1}=\gamma \left( E\left( \left\{ \tau \in Gp(x):\tau x\in
B_{n}\right\} \right) \right) ^{-1}
\end{equation*}%
or equivalently%
\begin{equation*}
E\left( \left\{ \rho \in Gr(\gamma ):\rho \gamma x\in B_{n}\right\} \right)
=E\left( \left\{ \tau \in Gp(x):\tau x\in B_{n}\right\} \right) \gamma ^{-1}
\end{equation*}%
Since $\tau \mapsto \tau \gamma ^{-1}$ is a homeomorphism from $Gp(x)$ to $%
Gr(\gamma )$ we have that 
\begin{eqnarray*}
E\left( \left\{ \tau \in Gp(x):\tau x\in B_{n}\right\} \right) \gamma ^{-1}
&=&E\left( \left\{ \tau \in Gp(x):\tau x\in B_{n}\right\} \gamma ^{-1}\right)
\\
&=&E\left( \left\{ \rho \in Gr(\gamma ):\rho \gamma x\in B_{n}\right\}
\right)
\end{eqnarray*}

\begin{claim-no}
$\varphi $ is injective
\end{claim-no}

Assume that $x,y\in X$ are such that $\varphi (x)=\varphi (y)$. Thus $%
p(x)=p(y)$ and for every $n\in \omega $%
\begin{equation*}
\left\{ \gamma \in Gp(x):\gamma x\in B_{n}\right\} \bigtriangleup \left\{
\gamma \in Gp(y):\gamma y\in B_{n}\right\}
\end{equation*}%
is meager. Thus $\forall ^{\ast }\gamma \in Gp(x)$, $\forall n\in \omega $, $%
\gamma x\in B_{n}$ iff $\gamma y\in B_{n}$. Thus for some $\gamma \in Gp(x)$%
, $\gamma x\in B_{n}$ iff $\gamma y\in B_{n}$ for all $n\in \omega $. This
implies that $\gamma x=\gamma y$ and hence $x=y$.

\section{Countable Borel groupoids\label{Section: countable}}

\subsection{Actions of inverse semigroups on Polish spaces\label{Subsection:
inverse semigroups}}

An \emph{inverse semigroup }is a semigroup $T$ such that every $t\in T$ has
a semigroup-theoretic inverse $t^{\ast }\in T$. This means that $t^{\ast }$
is the unique element of $T$ such that%
\begin{equation*}
tt^{\ast }t=t\text{ and }t^{\ast }tt^{\ast }=t^{\ast }\text{.}
\end{equation*}%
If $T$ is an inverse semigroup, then the set $E(T)$ of idempotent elements
is a commutative subsemigroup of $T$, and hence a semilattice; see~\cite[%
Proposition 2.1.1]{paterson_groupoids_1999}. In particular $E(T)$ has a
natural order defined by%
\begin{equation*}
e\leq f\text{ iff }ef=fe=e\text{.}
\end{equation*}%
Observe that for every $t\in T$ the elements $tt^{\ast }$ and $t^{\ast }t$
are idempotent.

Suppose that $X$ is a Polish space. The semigroup $\mathcal{H}(X)$ of
partial homeomorphisms between open subsets of $X$ is clearly an inverse
semigroup.

\begin{definition}
An action $\theta :T\curvearrowright X$ of a countable inverse semigroup $T$
on the Polish space $X$ is a semigroup homomorphism $\theta :t\mapsto \theta
_{t}$ from $T$ to $\mathcal{H}(X)$.
\end{definition}

Observe that a semigroup homomorphism between inverse semigroups
automatically preserves inverses; see~\cite[Proposition 2.1.1]%
{paterson_groupoids_1999}.

\subsection{Étale Polish groupoids\label{Subsection: ettale groupoids}}

Suppose that $G$ is a Polish groupoid. A subset $u$ of $G$ is a bisection if
the source and range maps restricted to $u$ are injective. A bisection of $G$
is open if it is an open subset of $G$. It is not difficult to verify that
the following conditions are equivalent:

\begin{enumerate}
\item The source and range maps of $G$ are local homeomorphisms from $G$ to $%
G^{0}$;

\item Composition of arrows in $G$ is a local homeomorphism from $G^{2}$ to $%
G$;

\item $G$ has a countable basis of open bisections;

\item $G$ has a countable inverse semigroup of open bisections that is basis
for the topology of $G$;

\item $G^{0}$ is an open subset of $G$.
\end{enumerate}

When these equivalent conditions are satisfied, $G$ is called étale Polish
groupoid. If $G$ is an étale Polish groupoid, then in particular for every $%
x\in G^{0}$ the fiber $Gx$ is a countable discrete subset of $G$.

\subsection{The groupoid of germs\label{Subsection: groupoid of germs}}

Suppose that $\theta :T\curvearrowright X$ is an action of a countable
inverse semigroup on a Polish space. We want to associate to such an action
an étale Polish groupoid $\mathcal{G}(\theta ,T,X)$ that contains all the
information about the action. This construction can be found in~\cite%
{exel_inverse_2008} in the case when $X$ is locally compact.

If $e\in E(T)$ denote by $D_{e}$ the domain of $\theta _{e}$. Observe that
the domain of $\theta _{t}$ is $D_{t^{\ast }t}$ and the range of $\theta
_{t} $ is $D_{tt^{\ast }}$. Define $\Omega $ to be the subset of $T\times X$
of pairs $\left( u,x\right) $ such that $x\in D_{u^{\ast }u}$. Consider the
equivalence relation $\sim $ on $\Omega $ defined by $\left( u,x\right) \sim
\left( v,y\right) $ iff $x=y$ and for some $e\in E(S)$, $ue=ve$ and $x\in
D_{e}$. The equivalence class $\left[ u,x\right] $ of $\left( u,x\right) $
is called the \emph{germ} of $u$ at $x$. Observe that if $e$ witnesses that $%
\left( u,x\right) \sim \left( v,x\right) $ then, after replacing $e$ with $%
u^{\ast }uv^{\ast }ve$ we can assume that $e\leq u^{\ast }u$ and $e\leq
v^{\ast }v$. It can be verified as in~\cite[Proposition 4.7]%
{exel_inverse_2008} that if $\left( u,x\right) $ and $\left( v,y\right) $
are in $\Omega $ and $x=\theta _{v}(y)$ then $\left( uv,y\right) \in \Omega $%
. Moreover the germ $\left[ uv,y\right] $ of $uv$ at $y$ depends only on $%
\left[ x,s\right] $ and $\left[ y,t\right] $.

One can then define the groupoid $\mathcal{G}(\theta ,T,X)=\Omega \left/
\sim \right. $ of germs of the action $S\curvearrowright X$ obtained by
setting

\begin{itemize}
\item $\mathcal{G}(\theta ,T,X)^{2}=\left\{ \left( \left[ u,x\right] ,\left[
v,y\right] \right) :\theta _{v}(y)=x\right\} $,

\item $\left[ u,x\right] \left[ v,y\right] =\left[ uv,y\right] $, and

\item $\left[ u,x\right] ^{-1}=\left[ \theta _{u^{\ast }},\theta _{u}(x)%
\right] $.
\end{itemize}

Observe that the map $x\mapsto \left[ e,x\right] $ from $X$ to $G$, where $e$
is any element of $E(S)$ such that $x\in D_{e}$, is a well-defined bijection
from $X$ to the set of objects $G^{0}$ of $G$. Identifying $X$ with $G^{0}$
we have that the source and range maps $s$ and $r$ are defined by 
\begin{equation*}
s\left[ u,x\right] =x
\end{equation*}%
and 
\begin{equation*}
r\left[ u,x\right] =\theta _{u}(x).
\end{equation*}%
We now define the topology of $\mathcal{G}(\theta ,T,X)$. For $u\in T$ and $%
U\subset D_{u^{\ast }u}$ open define%
\begin{equation*}
\Theta \left( u,U\right) =\left\{ \left[ u,x\right] \in G:x\in U\right\}
\end{equation*}%
It can be verified as in~\cite[Proposition 4.14, Proposition 4.15, Corollary
4.16, Proposition 4.17, and Proposition 4.18]{exel_inverse_2008} that the
following hold:

\begin{enumerate}
\item $\mathcal{G}(\theta ,T,X)$ is an étale Polish groupoid;

\item the map $x\mapsto \left[ e,x\right] $ where $e$ is any element of $%
E(S) $ such that $x\in D_{s}$, is a homeomorphism from $X$ onto the space of
objects of $\mathcal{G}\left( \theta ,T,x\right) $;

\item if $u\in T$ and $U\subset D_{u^{\ast }u}$ then $\Theta \left(
u,U\right) $ is an open bisection of $U$, and the map $x\mapsto \left[ u,x%
\right] $ is a homeomorphism from $U$ onto $\theta \left( u,U\right) $;

\item if $\mathcal{A}$ is a basis for the topology of $X$, then the
collection%
\begin{equation*}
\left\{ \theta \left( u,A\cap D_{u^{\ast }u}\right) :u\in S,A\in \mathcal{A}%
\right\}
\end{equation*}%
is a basis of open bisections for $\mathcal{G}(\theta ,T,X)$.
\end{enumerate}

\subsection{Regularity of the groupoid of germs\label{Subsection: regularity
germs}}

The groupoid of germs $\mathcal{G}(\theta ,T,X)$ for an action $\theta
:T\curvearrowright X$ is in general not Hausdorff, even when $X$ is locally
compact. Here we isolate a condition that ensures that $\mathcal{G}(\theta
,T,X)$ is regular.

Define the order $\leq $ on $T$ by setting $u\leq v$ iff $u=vu^{\ast }u$.
Observe that this extends the order of $E(T)$. Moreover if $u\leq v$ then%
\begin{equation*}
u^{\ast }u=v^{\ast }vu^{\ast }u^{\ast }v^{\ast }v=v^{\ast }vu^{\ast }u
\end{equation*}%
and hence $u^{\ast }u\leq v^{\ast }v$. We say that $T$ is a \emph{semilattice%
} if it is a semilattice with respect to the order $\leq $ just defined,
i.e.\ for every pair $u,v$ of elements of $T$ there is a largest element $%
u\wedge v$ below both $u$ and $v$.

\begin{proposition}
\label{Proposition: Polish germs}Suppose that $T$ is a semilattice. If there
is a subset $C$ of $T$ such that:

\begin{enumerate}
\item for every $u\in T$ and $x\in D_{u^{\ast }u}$ there is $c\in C$ such
that $x\in D_{\left( u\wedge c\right) ^{\ast }\left( u\wedge c\right) }$, and

\item for every distinct $c,d\in C$, $\Theta \left( c,D_{c^{\ast }c}\right)
\cap \Theta \left( d,D_{d^{\ast }d}\right) =\varnothing $,
\end{enumerate}

then the groupoid of germs $\mathcal{G}(\theta ,T,X)$ is regular.
\end{proposition}

\begin{proof}
Suppose that $\left[ u,x\right] $ is an element of $\mathcal{G}(\theta ,T,X)$%
, and $W$ is an open neighborhood of $\left[ u,x\right] $ in $\mathcal{G}%
(\theta ,T,X)$. There are an open subset $U$ of $X$ contained in $D_{u^{\ast
}u}$ such that $\left[ u,x\right] \in \Theta \left( u,U\right) \subset W$.
Pick $c\in C$ such that $x\in D_{\left( u\wedge c\right) ^{\ast }\left(
u\wedge c\right) }$, and an open neighborhood $V$ of $x$ whose closure $%
\overline{V}$ is contained in $U\cap D_{\left( u\wedge c\right) ^{\ast
}\left( u\wedge c\right) }$. We claim that $\Theta \left( u\wedge c,V\right) 
$ is an open neighborhood of $\left[ u,x\right] $ whose closure is contained
in $W$. To show this it is enough to show that $\Theta \left( u\wedge c,%
\overline{V}\right) $ is closed in $\mathcal{G}(\theta ,T,X)$. Pick $\left[
v,y\right] \in \mathcal{G}(\theta ,T,X)\backslash \Theta \left( u\wedge c,%
\overline{V}\right) $. If $y\notin \overline{V}$ then clearly there is an
open neighborhood of $\left[ t,y\right] $ disjoint from $\Theta \left(
u\wedge c,\overline{V}\right) $. Suppose that $y\in \overline{V}$. Pick $%
d\in C$ such that $y\in D_{\left( u\wedge d\right) ^{\ast }\left( u\wedge
d\right) }$. In such case we have that 
\begin{equation*}
\Theta \left( u\wedge d,D_{\left( u\wedge d\right) ^{\ast }\left( u\wedge
d\right) }\right)
\end{equation*}%
is an open neighborhood of $y$ disjoint from $\Theta \left( u\wedge c,%
\overline{V}\right) $. This concludes the proof.
\end{proof}

\subsection{Étale groupoids as groupoids of germs\label{Subsection: ettale
as germs}}

Suppose that $G$ is an étale Polish groupoid, and $\Sigma $ is a countable
inverse semigroup of open bisections of $G$. One can define the standard
action of $\Sigma $ on $G^{0}$ by setting $D_{e}=e$ for every $e\in E(\Sigma
)$, and $\theta _{u}:D_{u^{\ast }u}\rightarrow D_{uu^{\ast }}$ by 
\begin{equation*}
\theta _{u}(x)=r(ux)\text{,}
\end{equation*}%
where $ux$ is the only element of $u$ with source $x$. The same proof as 
\cite[Proposition 5.4]{exel_inverse_2008} shows the following fact:

\begin{proposition}
\label{Proposition: ettale are germs}Suppose that $\Sigma $ is a countable
inverse semigroup of open bisections of $G$ such that $\bigcup \Sigma =G$
and for every $u,v\in \Sigma $, $u\cap v$ is the union of the elements of $%
\Sigma $ contained in $u\cap v$. Consider the standard action $\theta
:\Sigma \curvearrowright G^{0}$. The map from $\mathcal{G}\left( \theta
,\Sigma ,X\right) $ to $G$ assigning to the germ $\left[ u,x\right] $ of $u$
at $x$ the unique element of $u$ with source $x$ is well defined, and it is
an isomorphism of étale Polish groupoids.
\end{proposition}

In particular every étale Polish groupoid is isomorphic to the groupoid of
germs of an action of an inverse semigroup on a Polish space.

\subsection{Borel bisections\label{Subsection: Borel slices}}

We will say that a (standard) Borel groupoid is countable if for every $x\in
G^{0}$, the set $Gx=s^{-1}\left[ \left\{ x\right\} \right] $ is countable.
Observe that the countable Borel equivalence relations are exactly the
principal countable Borel groupoids.

Suppose that $G$ is a countable Borel groupoid. Observe that the set $%
\mathcal{S}(G)$ of Borel bisections of $G$ is an inverse semigroup. The
idempotent semilattice $E(S)$ is the Boolean algebra of Borel subsets of $%
G^{0}$. The order $\leq $ on $\mathcal{S}(G)$ as in Subsection~\ref%
{Subsection: regularity germs} is defined by $u\leq v$ iff $u\subset v$.
Therefore $(S)$ is a semilattice with $u\wedge v=u\cap v$.

\begin{lemma}
\label{Lemma: countable-to-one}Suppose that $X,Z$ are standard Borel spaces
and $s:Z\rightarrow X$ is a Borel countable-to-one surjection. There is a
countable partition $\left( P_{n}\right) _{n\in \omega }$ of $Z$ into Borel
subsets such that $s_{|P_{n}}$ is 1:1 for every $n\in \omega $.
\end{lemma}

\begin{proof}
It is enough to show that $Z=\bigcup_{n}P_{n}$, where $P_{n}$ are Borel
subsets of $Z$ such that $s_{|P_{n}}$ is 1:1. After replacing $Z$ with the
disjoint union of $Z$ and $X\times \omega $, and setting $s\left( x,n\right)
=x$ for $\left( x,n\right) \in X\times \omega $, we can assume that for
every $x\in X$ the inverse image $s^{-1}\left\{ x\right\} $ is countably
infinite. We want to define a Borel function $e:X\rightarrow Z^{\omega }$
such that $\left\{ e(x)_{n}:n\in \omega \right\} $ is an enumeration of $%
s^{-1}\left\{ x\right\} $ for every $x\in X$. Consider the Borel subset $E$
of $X\times Z^{\omega }$ defined by%
\begin{eqnarray*}
\left( x,\left( e_{n}\right) \right) \in E\text{ } &\Leftrightarrow &\left(
e_{n}\right) \text{ is a enumeration of }s^{-1}\left\{ x\right\} \\
&\Leftrightarrow &\text{ }s\left( e_{n}\right) =x\text{ and }\forall z\in
s^{-1}\left\{ x\right\} \text{ }\exists n\text{ such that }z=e_{n}\text{.}
\end{eqnarray*}%
(Recall that the image of a standard Borel space under a countable-to-one
Borel function is Borel; see~\cite[Exercise 18.15]{kechris_classical_1995}.)
We want to find a Borel uniformization of $E$. For each $x\in X$ endow $%
s^{-1}\left\{ x\right\} $ with the discrete topology and $s^{-1}\left\{
x\right\} ^{\omega }$ with the product topology. Observe that for $\left(
e_{n}\right) \in s^{-1}\left\{ x\right\} ^{\omega }$ we have that $\left(
e_{n}\right) \in E_{x}$ iff $\forall z\in s^{-1}\left\{ x\right\} $ $\exists
n\in \omega $ such that $e_{n}=z$. Thus $E_{x}$ is a dense $G_{\delta }$
subset of $s^{-1}\left\{ x\right\} ^{\omega }$. Define the following $\sigma 
$-ideal $\mathcal{I}_{x}$ in $Z^{\omega }$: $A\in \mathcal{I}_{x}$ iff $%
A\cap E_{x}$ is meager in $s^{-1}\left\{ x\right\} ^{\omega }$. Thus $%
E_{x}\notin \mathcal{I}_{x}$. In order to conclude that $E$ has a Borel
uniformization, by~\cite[Theorem 18.6]{kechris_classical_1995} it is enough
to show that the assignment $x\mapsto \mathcal{I}_{x}$ is Borel-on-Borel as
in~\cite[Definition 18.5]{kechris_classical_1995}. Suppose that $Y$ is a
standard Borel space and $A\subset Y\times X\times Z^{\omega }$. Consider
the set 
\begin{eqnarray*}
&&\left\{ (y,x)\in Y\times X:A_{y,x}\in \mathcal{I}_{x}\right\} \\
&=&\left\{ (y,x)\in Y\times X:A_{y,x}\cap E_{x}\text{ is meager in }%
s^{-1}\left\{ x\right\} ^{\omega }\right\}
\end{eqnarray*}%
Clearly we can assume that $A\subset Y\times E$. If $e:\omega \rightarrow
s^{-1}\left\{ x\right\} $ is a bijection, then $e$ induces a homeomorphism $%
\pi _{e}:\omega ^{\omega }\rightarrow s^{-1}\left\{ x\right\} ^{\omega }$.
Therefore for $(y,x)\in Y\times X$ we have that 
\begin{eqnarray*}
A_{y,x}\cap E_{x}\text{ is meager in }s^{-1}\left\{ x\right\} ^{\omega }%
\text{ } &\Leftrightarrow &\text{ }\pi _{e}^{-1}\left[ A_{y,x}\cap
s^{-1}\left\{ x\right\} ^{\omega }\right] \text{ is meager} \\
&\Leftrightarrow &\left\{ w\in \omega ^{\omega }:\pi _{e}(w)\in
A_{y,x}\right\} \text{ is meager.}
\end{eqnarray*}%
Consider the Borel subset $Q$ of $Y\times X\times Z^{\omega }$ defined by $%
\left( y,x,e\right) \in Q$ iff $\left( x,e\right) \in E$ and $\forall n,m\in
\omega $ if $n\neq m$ then $e_{n}\neq e_{m}$ and $\left\{ w\in \omega
^{\omega }:\left( y,x,e\circ w\right) \in A_{y,x}\right\} $ is meager. We
have that%
\begin{eqnarray*}
A_{y,x}\in \mathcal{I}_{x}\text{ } &\Leftrightarrow &\text{ }\exists e\text{
such that }\left( y,x,e\right) \in Q \\
&\Leftrightarrow &\forall e\forall n\neq m\in \omega \text{, }\left(
x,e\right) \in E\text{, and }e_{n}\neq e_{m}\text{ }\Rightarrow \text{ }%
\left( z,x,e\right) \in Q\text{.}
\end{eqnarray*}%
This shows that $\left\{ (y,x):A_{y,x}\in \mathcal{I}_{x}\right\} $ is both
analytic and co-analytic, and hence Borel.
\end{proof}

\begin{proposition}
\label{Proposition: countable-to-one}If $G$ is a countable Borel groupoid,
then there is a countable partition of $G$ into Borel bisections. Moreover
for every $n\in \omega $ we have that%
\begin{equation*}
\left\{ x\in G^{0}:\left\vert Gx\right\vert =n\right\}
\end{equation*}%
is Borel.
\end{proposition}

\begin{proof}
The source map $s:G\rightarrow G^{0}$ satisfies the hypothesis of Lemma~\ref%
{Lemma: countable-to-one}. Therefore one can find a countable partition $%
\mathcal{H}$ of $G$ into Borel subsets such that the source map is 1:1 on
every element of $\mathcal{H}$. Define 
\begin{equation*}
\mathcal{C}=\left\{ u\cap v^{-1}:u,v\in \mathcal{H}\right\}
\end{equation*}%
an observe that $\mathcal{C}$ is a countable collection of pairwise disjoint
Borel bisections of $G$. Observe now that for every $u\in \mathcal{C}$,%
\begin{equation*}
\left\{ x\in G^{0}:\exists \gamma \in u\text{, }x=s(\gamma )\right\} =s\left[
u\right] =u^{-1}u
\end{equation*}%
is Borel being 1:1 image of a Borel set. Moreover $\left\vert Gx\right\vert
=m$ iff $\exists u_{0},\ldots ,u_{m-1}\in \mathcal{C}$ pairwise distinct
such that $x\in u_{i}u_{i}^{-1}$ for $i\in m$ and $\forall w\in \mathcal{C}$
if $x\in ww^{-1}$ then $w=u_{i}$ for some $i\in m$.
\end{proof}

Let us say that a Borel bisection $u$ is \emph{full} if $%
uu^{-1}=u^{-1}u=G^{0}$. It is clear from Proposition~\ref{Proposition:
countable-to-one} that if $G$ is a countable Borel groupoid, then there is a
partition of $G$ into full Borel bisections.

\subsection{A Polish topology on countable Borel groupoids}

In this subsection we observe that any countable Borel groupoid is Borel
isomorphic to a regular zero-dimensional étale Polish groupoid. Suppose that 
$G$ is a countable Borel groupoid. Pick a countable partition $C$ of $G$
into full Borel bisections and consider the smallest inverse subsemigroup of 
$T$ with the property that $u\cap v\in T$ whenever $u,v\in T$. Observe that $%
T$ is countable. By~\cite[Exercise 13.5]{kechris_classical_1995} there is a
zero-dimensional Polish topology $\tau ^{0}$ on $G^{0}$ generating the Borel
structure on $G^{0}$ such that $u^{-1}u$ is clopen for every $u\in T$.
Consider the standard action $\theta $ of $T$ on $\left( G^{0},\tau
^{0}\right) $ and observe that it satisfies the condition of Proposition~\ref%
{Proposition: Polish germs}. Therefore the associated groupoid of germs $%
\mathcal{G}\left( \theta ,T,G^{0}\right) $ is an étale zero-dimensional
regular Polish groupoid. Arguing as in the proof of~\cite[Proposition 5.4]%
{exel_inverse_2008} one can verify that the function $\phi $ from $G$ to $%
\mathcal{G}\left( \theta ,T,G^{0}\right) $ sending $\gamma $ to $\left[
c,s(\gamma )\right] $ where $c$ is the only element of $C$ such that $\gamma
\in C$ is a well defined Borel isomorphism of countable Borel groupoids.

\subsection{Treeable Borel groupoids\label{Subsection: treeable groupoids}}

Suppose that $G$ is a countable Borel groupoid. A $\emph{graphing}$ $Q$ of $%
G $ is a Borel subset $Q$ of $G\left\backslash G^{0}\right. $ such that $%
Q=Q^{-1}$ and $\bigcup_{n\in \omega }Q^{n}=G$, where $Q^{0}=G^{0}$. Suppose
that $Q$ is a graphing of $G$. Define $P^{\ast }(Q)$ to be the set of finite
nonempty sequences $(\gamma _{i})_{i\in n+1}$ in $Q$ such that $r(\gamma
_{i+1})=s(\gamma _{i})$ and $\gamma _{i+1}\neq \gamma _{i}^{-1}$ for $i\in n$%
. For $(\gamma _{i})_{i\in n+1}$ in $P^{\ast }(Q)$ one can define%
\begin{equation*}
\prod_{i\in n+1}\gamma _{i}
\end{equation*}%
to be the product $\gamma _{n}\gamma _{1n-2}\cdots \gamma _{1}\gamma _{0}$
in $G$. We say that $Q$ is a \emph{treeing }if for every $(\gamma
_{i})_{i\in n+1}\in P^{\ast }(Q)$, 
\begin{equation*}
\prod_{i\in n+1}\gamma _{i}\notin Q^{0}
\end{equation*}
or, equivalently, for every $\gamma \in G\backslash G^{0}$ there is exactly
one element $(\gamma _{i})_{i\in n+1}$ of $P^{\ast }(Q)$ such that $%
\prod_{i\in n+1}\gamma _{i}=\gamma $. A countable Borel groupoid is \emph{%
treeable} when it admits a treeing~\cite[Section 8]%
{anantharaman-delaroche_old_2011}.

It is not difficult to verify that a principal countable Borel groupoid is
treeable precisely when it is treeable as an equivalence relation. A
countable group is treeable as groupoid if and only if it is a free group.

In the following if $Q$ is a treeing of $G$ we denote by $P(Q)$ the union of 
$P^{\ast }(Q)$ and $\left\{ \varnothing \right\} $.\ In analogy with free
groups, if $\left( \gamma _{n},\ldots ,\gamma _{0}\right) \in P(Q)$ we say
that $\gamma _{n}\cdots \gamma _{0}$ is a \emph{reduced word},\emph{\ }and
that the length $l(\gamma _{n}\cdots \gamma _{0})$ of $\gamma _{n}\cdots
\gamma _{0}$ is $n+1$.

\begin{proposition}
Suppose that $G$ is a countable Borel groupoid. If there is a Borel complete
section $A$ for $E_{G}$ such that $G_{|A}$ is treeable, then $G$ is treeable.
\end{proposition}

\begin{proof}
Pick a Borel function $f:G^{0}\rightarrow G$ such that $f(a)=a$ for $a\in A$%
, $s\left( f(x)\right) =x$ and $r\left( f(x)\right) \in A$ for $x\in G^{0}$
. Suppose that $Q_{A}$ is a treeing for $G_{|A}$. Observe that $Q_{A}\cup f%
\left[ G^{0}\left\backslash A\right. \right] $ is a treeing for $G$.
\end{proof}

We want to show that Borel subgroupoids of treeable groupoid are treeable. A
particular case of this statement is that a subgroup of a countable free
group is free, which is the well known \emph{Nielsen-Schreier theorem}. The
strategy of our proof will be a Borel version for groupoids of Schreier's
proof of the Nielsen-Schreier theorem.

Suppose that $G$ is a treeable groupoid with no elements of order $2$, and $%
H $ is a Borel subgroupoid of $G$. In the rest of the subsection we will
show that $H$ is treeable. Denote by $\sim _{H}$ the equivalence relation $%
\gamma \sim _{H}\rho $ iff $\gamma H=\rho H$. Suppose that $Q$ is a treeing
for $G$. Since $G$ has no elements of order $2$ we can write $Q=Q^{+}\cup
Q^{-}$ where $Q^{+}$ and $Q^{-}$ are disjoint and $Q^{+}=\left( Q^{-}\right)
^{-1}$. A Borel transversal $U$ for $\sim _{H}$ is \emph{Schreier} if $%
\gamma _{n}\cdots \gamma _{0}\in T$ implies $\gamma _{k}\cdots \gamma
_{0}\in T$ for $k\in n$. We want to show that there is a Schreier Borel
transversal for $H$.

Suppose that $\left( V_{n}\right) _{n\in \omega }$ is a partition of $%
G\backslash G^{0}$ into full Borel bisections. If $\gamma _{n}\cdots \gamma
_{0}$ and $\gamma _{m}^{\prime }\cdots \gamma _{0}^{\prime }$ are reduced
words with $r(\gamma _{n})=r\left( \gamma _{m}^{\prime }\right) =x$, set%
\begin{equation*}
\gamma _{n}\cdots \gamma _{0}<_{x}\gamma _{m}^{\prime }\cdots \gamma
_{0}^{\prime }
\end{equation*}%
iff $n<m$, or $n=m$ and for some $k\in n$, $\gamma _{i}=\gamma _{i}^{\prime
} $ for $i\in k$ and for some $N\in \omega $, $\gamma _{k}\in V_{N}$ while $%
\gamma _{k}^{\prime }\notin V_{n}$ for any $n\leq N$. Define also%
\begin{equation*}
x<_{x}\gamma _{n}\cdots \gamma _{0}\text{.}
\end{equation*}%
Observe that $<_{x}$ is a Borel order of $xG$ with minimum $x$, and the
function $x\mapsto <_{x}$ is Borel. Define now for $\gamma \in G$, $%
\overline{\gamma }$ to be the $<_{r(\gamma )}$-least element of $\gamma H$.
Thus $\overline{\gamma }\in \gamma H$ and hence $\overline{\gamma }%
^{-1}\gamma \in H$. Consider $U=\left\{ \overline{\gamma }^{-1}\gamma
:\gamma \in G\right\} $ and observe that, since $x$ is the $<_{x}$-minimum
element of $xG$, $U\cap H\subset H^{0}$. Arguing as in~\cite[Section 2.3]%
{johnson_presentations_1997} one can show that $U$ is a Schreier transversal
for $\sim _{H}$. Define then 
\begin{equation*}
A=\left\{ \overline{\gamma u}^{-1}\gamma u:u\in U\text{, }\gamma \in
Q\right\} \subset H.
\end{equation*}%
The same proof as Lemma~3 in~\cite[Section 3.3]{johnson_presentations_1997}
shows that $\bigcup_{n\in \omega }A^{n}=H$. Define now 
\begin{equation*}
B=\left\{ \overline{\gamma u}^{-1}\gamma u:u\in U\text{, }\gamma \in Q^{+}%
\text{, and }\gamma u\notin U\right\} \text{.}
\end{equation*}%
The same proof as Lemma~4 in~\cite[Section 3.4]{johnson_presentations_1997}
shows that 
\begin{equation*}
B^{-1}=\left\{ \overline{\gamma u}^{-1}\gamma u:u\in U\text{, }\gamma \in
Q^{-}\text{, and }\gamma u\notin U\right\} \text{,}
\end{equation*}%
and $A\backslash H^{0}$ is the disjoint union of $B$ and $B^{-1}$. Finally
one can show that $A\backslash H^{0}$ is a treeing for $G$ as in~\cite[%
Section 3.6]{johnson_presentations_1997}. The proof is the same as the proof
of Theorem~1 in~\cite[Section 3.6]{johnson_presentations_1997}. The
fundamental lemma is the following:

\begin{lemma}
\label{Lemma: treeable subgroupoid}Suppose that $b=\overline{u\gamma }%
^{-1}u\gamma \in A\backslash H^{0}$ and $b^{\prime }=\overline{v\rho }%
^{-1}v\rho \in A\backslash H^{0}$. The product $\rho v\overline{\gamma u}%
^{-1}\gamma $ is equal to a reduced word $\rho w\gamma $ for some $w\in G$,
unless $v=\overline{\gamma u}$ and $\rho =\gamma ^{-1}$, in which case%
\begin{equation*}
u=\overline{\gamma ^{-1}\overline{\gamma u}}=\overline{\rho v}
\end{equation*}%
and%
\begin{equation*}
b^{\prime }=b^{-1}\text{.}
\end{equation*}
\end{lemma}

The proof of Lemma~\ref{Lemma: treeable subgroupoid} is analogous to the
proof of Lemma~5 in~\cite[Section 3.5]{johnson_presentations_1997}.

\section{Functorial Borel complexity and treeable equivalence relations\label%
{Section: functorial reducibility and treeable}}

\subsection{The lifting property}

\begin{definition}
\label{Definition: lifting property}Suppose that $G$ is a Polish groupoid.
We say that $G$ has the \emph{lifting property} if the following holds: For
any Polish groupoid $H$ such that $E_{H}$ is Borel, and any Borel function $%
f:G^{0}\rightarrow H^{0}$ such that $f(x)E_{H}f(x^{\prime })$ whenever $%
xE_{G}x^{\prime }$, there is a Borel functor $F:G\rightarrow H$ that extends 
$f$.
\end{definition}

\begin{remark}
\label{Remark: lifting-lifting}If $E_{G}$ has the lifting property (as a
principal groupoid), then $G$ has the lifting property.
\end{remark}

\begin{proposition}
\label{Proposition: treeable implies lifting}A treeable countable Borel
groupoid with no elements of order $2$ has the lifting property.
\end{proposition}

\begin{proof}
Suppose that $G$ is a treeable countable Borel groupoid with no elements of
order $2$, $H$ is a Polish groupoid such that $E_{H}$ is Borel, and$\
f:G^{0}\rightarrow H^{0}$ is a Borel function such that $f(x)E_{G}f(x^{%
\prime })$ whenever $xE_{G}x^{\prime }$. Suppose that $Q$ is a treeing for $%
G $. Write $Q=Q^{+}\cup Q^{-}$ where $Q^{+}=\left( Q^{-}\right) ^{-1}$ and $%
Q^{+}$ and $Q^{-}$ are disjoint. Since $E_{H}$ is Borel, then map $%
(x,y)\mapsto xHy$ from $E_{G}$ to $F(H)\backslash \left\{ \varnothing
\right\} $ is Borel by Theorem~\ref{Theorem: Borel orbit equivalence
relations}. Fix a Borel map $\sigma :F(H)\backslash \left\{ \varnothing
\right\} \rightarrow H$ such that $\sigma (A)\in A$ for every $A\in
F(H)\backslash \left\{ \varnothing \right\} $. Define

\begin{itemize}
\item $F(x)=f(x)$ for $x\in G^{0}$,

\item $F(\gamma )=\sigma \left( f(r(\gamma ))Hf(s(\gamma ))\right) $ for $%
\gamma \in Q^{+}$,

\item $F(\gamma )=F\left( \gamma ^{-1}\right) ^{-1}$ for $\gamma \in \left(
Q^{+}\right) ^{-1}$, and

\item $F(\gamma _{n}\cdots \gamma _{0})=F(\gamma _{n})\cdots F(\gamma _{0})$
if $\gamma _{n}\cdots \gamma _{0}\in G\backslash G^{0}$ is a reduced word.
\end{itemize}

It is immediate to check that $F$ is a Borel functor such that $F_{|G^{0}}=f$%
.
\end{proof}

\begin{proposition}
\label{Proposition: extending lifting}If $G$ is a Polish groupoid and $%
A\subset G^{0}$ is a Borel complete section for $E_{G}$ such that $G_{|A}$
has the lifting property and there is a Borel map $\phi :G^{0}\rightarrow G$
such that $s\left( \phi (x)\right) =x$ and $r\left( \phi (x)\right) \in A$
for every $x\in G^{0}$, then $G$ has the lifting property.
\end{proposition}

\begin{proof}
Without loss of generality we can assume that $\phi (x)=x$ for $x\in A$.
Define $y(x)=r\left( \phi (x)\right) $ for $x\in G^{0}$. Suppose that $%
f:G^{0}\rightarrow H^{0}$ is a Borel function such that $f(x)E_{H}f(x^{%
\prime })$ whenever $xE_{G}x$. Since $G_{|A}$ ha the lifting property there
is a Borel functor $F:G_{|A}\rightarrow H$ such that $F_{|A}=f_{|A}$. Define 
$h(x)=\sigma \left( f\left( y(x)\right) Hf(x)\right) $. Define now for $\rho
\in G$ such that $s(\rho )=x$ and $r(\rho )=y$%
\begin{equation*}
F(\rho )=h(y)^{-1}F\left( \phi (y)\rho \phi (x)^{-1}\right) h(x)
\end{equation*}%
and observe that $F$ is a Borel functor such that $F_{|G^{0}}=f$.
\end{proof}

\begin{theorem}
\label{Theorem: essentially countable}Suppose that $G$ is a Polish groupoid.
If $E_{G}$ is essentially treeable, then $E_{G}$ has the lifting property.
\end{theorem}

\begin{proof}
Observe that the assignment $\left[ x\right] _{E_{G}}\mapsto I_{\left[ x%
\right] _{E_{G}}}$,where 
\begin{equation*}
A\in I_{\left[ x\right] _{E_{G}}}\Leftrightarrow \left\{ \gamma \in
xG:s(\gamma )\in A\right\} \text{ is meager}
\end{equation*}%
is a Borel ccc assignment of $\sigma $-ideals in the sense of~\cite[page 285]%
{kechris_countable_1992}; see Subsection~\ref{Subsection: Vaught transform}.
It follows from~\cite[Theorem 1.5]{kechris_countable_1992} together with the
fact that $E_{G}$ is essentially treeable that there is a countable Borel
subset $A$ of $G^{0}$ meeting every orbit in a countable nonempty set. Thus $%
\left( E_{G}\right) _{|A}$ is treeable equivalence relation. In particular
by Proposition~\ref{Proposition: treeable implies lifting} the equivalence
relation $\left( E_{G}\right) _{|A}$ has the lifting property. Therefore $%
G_{|A}$ has the lifting property. Since $\left( E_{G}\right) _{|A}$ is
countable one can find a Borel map $p:X\rightarrow A$ such that $\left(
x,p(x)\right) \in E_{G}$ for every $x\in X$ and $p(x)=x$ for $x\in A$. It
follows from Proposition~\ref{Proposition: extending lifting} that $E_{G}$
has the lifting property.
\end{proof}

\begin{corollary}
Suppose that $G$ and $H$ are Polish groupoids. If $E_{G}$ is essentially
treeable, and $E_{H}$ is Borel, then $G\leq _{B}H$ if and only if $E_{G}\leq
E_{H}$.
\end{corollary}

\begin{proposition}
Suppose that $G$ is a Polish groupoid. If $E_{G}$ is essentially countable,
then there is an invariant dense $G_{\delta }$ set $C\subset G^{0}$ such
that $\left( E_{G}\right) _{|C}$ is essentially hyperfinite.
\end{proposition}

\begin{proof}
By~\cite[Theorem 6.2]{hjorth_borel_1996} there is a comeager and invariant
subset $C_{0}$ of $G^{0}$ such that $\left( E_{G}\right) _{|C_{0}}$ is
essentially hyperfinite. Pick a dense $G_{\delta }$ subset $C_{1}$ of $C_{0}$
and then define%
\begin{equation*}
C=\left\{ x\in X:\forall ^{\ast }\gamma \in Gx\text{, }\gamma x\in
C_{1}\right\} \text{.}
\end{equation*}%
The properties of the Vaught transform together with Lemma~\ref{Lemma:
Kuratowski-Ulam} imply that $C$ is an invariant dense $G_{\delta }$ set
contained in $C_{0}$. In particular $\left( E_{G}\right) _{|C}$ is
essentially hyperfinite.
\end{proof}

\begin{corollary}
Suppose that $G$ is a Polish groupoid such that $E_{G}$ is essentially
countable. There is an invariant dense $G_{\delta }$ subset $C$ of $G^{0}$
with the following property: For any Polish groupoid $H$ such that $%
E_{G}\leq _{B}E_{H}$ and $E_{H}$ is Borel, $G_{|C}\leq _{B}H$.
\end{corollary}

\subsection{The cocycle property}

\begin{definition}
\label{Definition: cocycle property}An analytic groupoid $G$ has the \emph{%
cocycle property} if there is a Borel functor $F:E_{G}\rightarrow G$ such
that $F(x,x)=x$ for every $x\in G^{0}$.
\end{definition}

It is immediate to verify that a Polish group action $G\curvearrowright X$
has the cocycle property as defined in~\cite{hjorth_borel_1996} if and only
if the action groupoid $G\ltimes X$ has the cocycle property as in
Definition~\ref{Definition: cocycle property}. The proof of the following
proposition is essentially the same of the proof of the implication (ii)$%
\Rightarrow $(iii) in~\cite[Theorem 3.7]{jackson_countable_2002}, and it is
presented for convenience of the reader.

\begin{proposition}
\label{Proposition: treeable free}Suppose that $G$ is a countable Borel
groupoid, and $X$ a Borel $G$-space. If $G\ltimes X$ has the cocycle
property, then there is a free Borel $G$-space $Y$ such that $E_{G}^{Y}\sim
_{B}E_{G}^{X}$. Moreover if $G$ is treeable then $E_{G}^{X}$ is treeable.
\end{proposition}

\begin{proof}
Since $G\ltimes X$ has the cocycle property there is a Borel functor%
\begin{equation*}
F:E_{G}^{X}\rightarrow G
\end{equation*}%
such that $s\left( F(x,y)\right) =p(y\mathbb{)}$ and $F(x,y)y=x$. Consider
the equivalence relation $\sim $ on $G\ltimes X$ defined by $(\gamma ,x)\sim
(\rho ,y)$ iff $(x,y)\in E_{G}^{X}$ and $\gamma F(x,y)=\rho $. Clearly $\sim 
$ is Borel. We now show that it has a Borel selector. Observe that the range 
$H$ of $F$ is a Borel subgroupoid of $G$ (since $F$ is countable to one). By
Proposition~\ref{Proposition: coset selector} there is a Borel selector $%
t:G\rightarrow G$ for the equivalence relation $\gamma \sim _{H}\gamma
^{\prime }$ iff $\gamma H=\gamma ^{\prime }H$. Observe that if $(\gamma
,x)\sim (\rho ,y)$ then $\gamma H=\rho H$ and hence $t(\gamma )=t(\rho )$.
Moreover there is a unique element $x_{0}$ of $X$ such that $\left( t(\gamma
),x_{0}\right) \sim (\gamma ,x)$. Define $S(\gamma ,x)=\left( t(\gamma
),x_{0}\right) $ and observe that $S$ is a Borel selector for the
equivalence relation $\sim $. Define $Y$ to be the quotient of $G\ltimes X$
by $\sim $. Define now the Borel action of $G$ on $Y$ by $p\left[ \gamma ,x%
\right] =r(\gamma )$ and $\rho \left[ \gamma ,x\right] =\left[ \rho \gamma ,x%
\right] $ for $\rho \in Gr(\gamma )$. It is easy to verify that such an
action is free, and $\left[ \gamma ,x\right] E_{G}^{Y}\left[ \rho ,y\right] $
iff $xE_{G}^{X}y$. Let us now observe that $E_{G}^{X}\sim _{B}E_{G}^{Y}$. If 
$q:X\rightarrow G$ is a Borel map such that $s\left( q(x)\right) =p(x)$ for
every $x\in X$, then the map $x\mapsto \left[ q(x),x\right] $ is a Borel
reduction from $E_{G}^{X}$ to $E_{G}^{Y}$. Conversely the map $\left[ \gamma
,x\right] \mapsto x^{\ast }$ where $\left[ t(\gamma ),x^{\ast }\right]
=S(\gamma ,x)$ is a Borel reduction from $E_{G}^{Y}$ to $E_{G}^{X}$. Suppose
finally that $G$ is treeable with treeing $Q$. We want to show that $%
E_{G}^{X}$ is treeable. Since $E_{G}^{X}\sim _{B}E_{G}^{Y}$, it is enough to
show that $E_{G}^{Y}$ is treeable. Fix an equivalence class $\left[ \left[
\gamma ,x\right] \right] _{F}$ of $E_{G}^{Y}$. Observe that the map from $%
\left[ \left[ \gamma ,x\right] \right] _{E_{G}^{Y}}$ to $Gp(x)$ defined by $%
\left[ \rho ,y\right] \mapsto \rho F(y,x)$ is bijective. One can then
consider the treeing 
\begin{equation*}
\left\{ \left[ \rho ,y\right] \in Y:\rho F(y,x)\in Q\right\}
\end{equation*}%
for $E_{G}^{Y}$.
\end{proof}

Lemma~\ref{Lemma: extending cocycle} can be proved similarly as Proposition~%
\ref{Proposition: extending lifting}.

\begin{lemma}
\label{Lemma: extending cocycle}Suppose that $G$ is a countable Borel
groupoid action, and $A\subset G^{0}$ is a Borel complete section for $E_{G}$%
. If there is a Borel function $\psi :\left( E_{G}\right) _{|A}\rightarrow G$
such that $s\left( \psi \left( x,y\right) \right) =y$ and $r\left( \psi
\left( x,y\right) \right) =x$, then $G$ has the cocycle property.
\end{lemma}

\begin{lemma}
\label{Lemma: reduction implies cocycle}Suppose that $G$ is a countable
Borel groupoid. If $E_{G}\leq _{B}G$, then there is an invariant Borel
subset $Y$ of $G^{0}$ such that $G_{|Y}$ has the cocycle property and $%
\left( E_{G}\right) _{|Y}\sim _{B}E_{G}$
\end{lemma}

\begin{proof}
Since $E_{G}\leq _{B}G$ there is a Borel functor $F:E_{G}\rightarrow G$.
Define $A\subset G^{0}$ to be image of $X$ under $F$. Since $F_{|X}$ is
countable-to-one, $A$ is a Borel subset of $G^{0}$; see~\cite[Theorem 18.10]%
{kechris_classical_1995}. By~\cite[Exercise 18.14]{kechris_classical_1995}
there is a Borel function $g:A\rightarrow X$ such that $\left( f\circ
g\right) (y\mathbb{)}=y$ for every $y\in A$. Define now $Y$ to be the union
of the orbits of $G$ that meet $A$. Clearly $E_{G}\sim _{B}E_{G_{|Y}}$. The
map%
\begin{eqnarray*}
\left( E_{G}\right) _{|A} &\rightarrow &G \\
\left( x,y\right) &\mapsto &F\left( g(x\mathbb{)},g(y\mathbb{)}\right)
\end{eqnarray*}%
together with Lemma~\ref{Lemma: extending cocycle} imply that $G_{|Y}$ has
the cocycle property.
\end{proof}

\subsection{Free actions of treeable groupoids}

We want to show that, if $G$ is a treeable groupoid, and $G\curvearrowright
X $ is a free Borel action of $G$, then the associated orbit equivalence
relation is treeable. This will follow from a more general result about $%
\mathcal{L}$-structured equivalence relations.

Suppose that $\mathcal{L}=\left\{ R_{n}:n\in \omega \right\} $ is a
countable relational language in first order logic, where $R_{n}$ has arity $%
k_{n}\in \omega $. Suppose that $E$ is a countable Borel equivalence
relation on a standard Borel space $X$. According to~\cite[Definition 2.17]%
{jackson_countable_2002} the equivalence relation $E$ is $\mathcal{L}$%
-structured if there are Borel relations $R_{n}^{E}\subset X^{k_{n}}$ such
that, for any $k\in \omega $ and $x_{1},\ldots ,x_{k_{n}}\in X$, $%
x_{1},\ldots ,x_{k_{n}}$ belong to the same $E$-class whenever $\left(
x_{1},\ldots ,x_{k_{n}}\right) \in R_{n}^{E}$. In particular every $E$-class 
$\left[ x\right] $ is the universe of an $\mathcal{L}$-structure 
\begin{equation*}
\left\langle \left[ x\right] ,\left( \left[ x\right] ^{k_{n}}\cap
R_{n}^{E}\right) _{n\in \omega }\right\rangle \text{.}
\end{equation*}

Similarly, if $X$ is a standard Borel space, then \emph{standard Borel
bundle }$\mathcal{A}$\emph{\ of countable }$\mathcal{L}$\emph{-structures}
over $X$ is a standard Borel space $\mathcal{A}$ fibred over $X$ with
countable fibers $\left( A_{x}\right) _{x\in X}$, endowed with Borel subsets 
$R_{n}^{\mathcal{A}}\subset \mathcal{A}^{k_{n}}$ such that, for any $k\in
\omega $ and $a_{1},\ldots ,a_{k_{n}}\in \mathcal{A}$, $a_{1},\ldots
,a_{k_{n}}$ belong to the same fiber over $X$ whenever $\left( a_{1},\ldots
,a_{k_{n}}\right) \in R_{n}^{\mathcal{A}}$. Suppose that $G$ is a Polish
groupoid, and $\mathcal{A}$ is a standard Borel of $\mathcal{L}$-structures
over $G^{0}$. A Borel action $G\curvearrowright \mathcal{A}$ is a Borel map $%
\left( \gamma ,a\right) \mapsto \gamma a$ defined for $\left( \gamma
,a\right) \in \mathcal{A}\times G$ such that $a\in A_{s(\gamma )}$, and $%
(a_{1},\ldots ,a_{k_{n}})\in R_{n}^{\mathcal{A}}$ if and only if $(\gamma
a_{1},\ldots ,\gamma a_{k_{n}})\in R_{n}^{\mathcal{A}}$ for every $n\in
\omega $, $\gamma \in G$, and $a_{1},\ldots ,a_{k_{n}}\in A_{s(\gamma )}$.
The proof of the following theorem is very similar to the argument at the
beginning of Section 3.2 in~\cite{jackson_countable_2002}, and it is
reproduced for convenience of the reader.

\begin{theorem}
\label{Theorem: bundle of structures}Suppose that $G$ is a Polish groupoid
such that there is a standard Borel bundle $\mathcal{A}$ of countable $%
\mathcal{L}$-structures and a Borel action $G\curvearrowright \mathcal{A}$
such that the corresponding orbit equivalence relation $E_{G}^{\mathcal{A}}$
is smooth, and for every $a\in \mathcal{A}$ the stabilizer $G_{a}=\left\{
\gamma \in Gp_{\mathcal{A}}(a):\gamma a=a\right\} $ is a compact subset of $%
G $. If $X$ is a standard Borel space, and $G\curvearrowright X$ is a free
Borel action of $G$ on $X$, then there is an $\mathcal{L}$-structured
countable Borel equivalence relation $E$ such that $E\sim _{B}E_{G}^{X}$,
and moreover every class of $E$ is isomorphic to some fiber of $\mathcal{A}$.
\end{theorem}

\begin{proof}
By Corollary~\ref{Corollary: Borel selector} there is a Borel selector $t$
for $E_{G}^{\mathcal{A}}$. Moreover by Theorem~\ref{Theorem: Becker-Kechris
Borel} we can assume without loss of generality that the action $%
G\curvearrowright X$ is continuous. Define 
\begin{equation*}
\mathcal{A}\ast X=\left\{ (x,a)\in X\times A:a\in A_{p(x\mathbb{)}}\right\}
\end{equation*}%
and the action $G\curvearrowright \left( X\ast A\right) $ by $\gamma \left(
x,a\right) =\left( \gamma x,\gamma a\right) $. Observe that such an action
is free and in particular the associated orbit equivalence relation $\sim $
is Borel. We now show that $\sim $ has a Borel selector. Suppose that $%
\left( x,a\right) \in X\ast A$. Observe that if $\left( x,a\right) \sim
\left( x,b\right) $ then $aE_{G}^{A}b$ and hence $t(a)=t\left( b\right) $.
Therefore $t(a)$ depends only on the $\sim $-class $\left[ x,a\right] $ of $%
\left( x,a\right) $. Observe that%
\begin{equation*}
G_{t(a)}x=\left\{ y\in X:\left( y,t(a)\right) \in \left[ x,a\right] \right\}
\end{equation*}%
is a compact subset of $X$. Denote by $\sigma :F(X)\backslash \left\{
\varnothing \right\} \rightarrow X$ a Borel function such that $\sigma
\left( A\right) \in A$ for every $A\in F(X)\backslash \left\{ \varnothing
\right\} $. Define $S\left( x,a\right) =\left( \sigma \left(
G_{t(a)}x\right) ,t(a)\right) $, and observe that $S$ is a Borel selector
for $\sim $. Define the standard Borel space $Y=\left( X\ast A\right) /\sim $
and the countable Borel equivalence relation $E$ on $Y$ by $\left[ x,a\right]
E\left[ y,b\right] $ iff $xE_{G}^{X}y$. We now define for every $E$-class $C=%
\left[ \left[ x,a\right] \right] _{E}$ an $\mathcal{L}$-structure on $C$.
Fix $n\in \omega $ and suppose that $\left[ \gamma _{i}x,a_{i}\right] $ for $%
i\in k_{n}$ are elements of $C$, where $\gamma _{i}\in Gp(x\mathbb{)}$ for $%
i\in k_{n}$. Set $\left( \left[ \gamma _{i}x,a_{i}\right] \right) _{i\in
k_{n}}\in R_{n}^{C}$ iff $\left( \gamma _{i}^{-1}a_{i}\right) _{i\in
k_{n}}\in R_{n}^{\mathcal{A}}$. Using the fact that $G$ acts by $\mathcal{L}$%
-isomorphisms one can verify that this does not depend on the choice of $%
\left[ x,a\right] \in C$. Define now $\left( \left[ x_{i},a_{i}\right]
\right) _{i\in k_{n}}\in R_{n}^{E}$ if and only if $\left[ x_{0},a_{0}\right]
E\left[ x_{1},a_{1}\right] E\cdots E\left[ x_{k_{n}-1},a_{k_{n}-1}\right] $
and $\left( \left[ x_{i},a_{i}\right] \right) _{i\in k_{n}}\in R_{n}^{C}$
where $C=\left[ \left[ x,a\right] \right] _{E}$. This defines Borel
relations $R_{n}^{E}$ on $E$ that make $E$ $\mathcal{L}$-structured.
Moreover the Borel map $f:C\rightarrow \mathcal{A}$ defined by $f\left[
\gamma x,a\right] =\gamma ^{-1}a$, where $C$ is the class of $\left[ x,a%
\right] $, shows that the $\mathcal{L}$-structure 
\begin{equation*}
\left\langle C,\left( R_{n}\cap C^{k_{n}}\right) _{n\in \omega }\right\rangle
\end{equation*}%
is isomorphic to $A_{p(x\mathbb{)}}$. Finally we observe now that $E\sim
E_{G}^{X}$. If $q:G^{0}\rightarrow \mathcal{A}$ such that $q(x\mathbb{)}$
belongs to the fiber $A_{x}$ for every $x\in G^{0}$, then the map $x\mapsto %
\left[ x,q\left( p(x\mathbb{)}\right) \right] $ is a Borel reduction from $%
E_{G}^{X}$ to $E$. Conversely the map $\left[ x,a\right] \mapsto x^{\ast }$
where $\left[ x^{\ast },a^{\ast }\right] =S\left( \left[ x,a\right] \right) $
witnesses that $E$ is Borel reducible to $E_{G}^{X}$.
\end{proof}

Let us now consider the particular case of Theorem~\ref{Theorem: bundle of
structures} when $\mathcal{L}$ is the language with a single binary
relation. Assume further that $G$ is a treeable Borel groupoid. A standard
Borel bundle of trees over $X$ is a standard Borel bundle $\left(
A_{x}\right) _{x\in X}$ of countable $\mathcal{L}$-structures such that $%
A_{x}$ is a tree for every $x\in X$. A treeing of $G$ defines on $G$ a
structure of standard Borel bundle of trees over $G$. Moreover the action of 
$G$ on itself by left translation is compatible with such a bundle of trees
structure, and has a smooth orbit equivalence relation. Therefore by Theorem~%
\ref{Theorem: bundle of structures} $E$ is treeable. This concludes the
proof of the following corollary.

\begin{corollary}
\label{Corollary: free action treeable is treeable}If $G$ is a countable
treeable groupoid, and $G\curvearrowright X$ is a free Borel action, then
the orbit equivalence relation $E_{G}^{X}$ is treeable.
\end{corollary}

\subsection{Characterizing treeable equivalence relations}

Denote by $\mathbb{F}_{\infty }$ the free countable group on infinitely many
generators. The following result subsumes~\cite[Theorem 3.7]%
{jackson_countable_2002}.

\begin{theorem}
\label{Theorem: non-treeable}Suppose that $E$ is a countable Borel
equivalence relation on a standard Borel space $X$. The following statements
are equivalent:

\begin{enumerate}
\item $E$ is treeable;

\item $E$ has the lifting property;

\item For every countable Borel groupoid $G$ and Borel action $%
G\curvearrowright X$ such that $E_{G}^{X}=E$, the groupoid $G\ltimes X$ has
the cocycle property;

\item For every Borel action $\mathbb{F}_{\infty }\curvearrowright X$ such
that $E_{\mathbb{F}_{\infty }}^{X}=E$, $E_{\mathbb{F}_{\infty }}^{X}\leq _{B}%
\mathbb{F}_{\infty }\ltimes X$;

\item For every countable Borel groupoid $G$ and Borel action $%
G\curvearrowright X$ such that $E_{G}^{X}=E$, there is a free Borel action $%
G\curvearrowright Y$ such that $E_{G}^{Y}\sim E_{G}^{X}$;

\item For every countable Borel groupoid $G$ and Borel action $%
G\curvearrowright X$ such that $E\subset E_{G}^{X}$ there is a free Borel
action $G\curvearrowright Y$ such that $E\sqsubseteq _{B}E_{G}^{Y}$.
\end{enumerate}
\end{theorem}

\begin{proof}
\begin{description}
\item[(1)$\Rightarrow $(2)] It follows from Proposition~\ref{Proposition:
treeable implies lifting}.

\item[(2)$\Rightarrow $(3)] It follows form the fact that if $E_{G}^{X}$ has
the lifting property, then $G\ltimes X$ has the cocycle property.

\item[(3)$\Rightarrow $(4)] Obvious.

\item[(4)$\Rightarrow $(1)] It follows from Lemma~\ref{Lemma: reduction
implies cocycle} and Proposition~\ref{Proposition: treeable free}.

\item[(3)$\Rightarrow $(5)] This follows from Proposition~\ref{Proposition:
treeable free}.

\item[(5)$\Rightarrow $(1)] Consider an action $\mathbb{F}_{\infty
}\curvearrowright X$ such that $E=E_{\mathbb{F}_{\infty }}^{X}$ and then
apply (5) and Corollary~\ref{Corollary: free action treeable is treeable}.

\item[(2)$\Rightarrow $(6)] Since $E$ has the lifting property, there is a
Borel function $F:E\rightarrow G$ such that $s\left( F\left( x,y\right)
\right) =p(y\mathbb{)}$ and $F(x,y)y=x$ for every $(x,y)\in E$. Consider on $%
G\ltimes X$ the equivalence relation $(\gamma ,x)\sim (\rho ,y)$ iff $xEy$
and $\rho =\gamma F(x,y)$. Proceeding as in the proof of Proposition~\ref%
{Proposition: treeable free} one can show that $\sim $ has a Borel selector.
Thus the quotient $Y$ of $G\ltimes X$ by $\sim $ is standard. Define the
Borel action $G\curvearrowright Y$ by $p\left[ \gamma ,x\right] =r(\gamma )$
and $\rho \left[ \gamma ,x\right] =\left[ \rho \gamma ,x\right] $. As in the
proof of Proposition~\ref{Proposition: treeable free} one can show that such
an action is free. Moreover the map $x\mapsto \left[ p(x),x\right] $ is an
injective Borel reduction from $E$ to $E_{G}^{Y}$.

\item[(6)$\Rightarrow $(1)] It follows from Corollary~\ref{Corollary: free
action treeable is treeable} together with the fact that a subrelation of a
treeable equivalence relation is treeable~\cite[Proposition 3.3]%
{jackson_countable_2002}; see also Subsection~\ref{Subsection: treeable
groupoids}.
\end{description}
\end{proof}

The following corollary is an immediate consequence of the implication (4)$%
\Rightarrow $(1) in Theorem \ref{Theorem: non-treeable}.

\begin{corollary}
\label{Corollary: different complexity}For any nontreeable equivalence
relation $E$ there are countable Borel groupoids $G$ and $H$ that have $E$
as orbit equivalence relation such that $G$ is not Borel reducible to $H$.
Moreover one can take $G=E$ and $H$ to be the action groupoid of an action
of $\mathbb{F}_{\infty }$.
\end{corollary}

Corollary~\ref{Corollary: different complexity} can be interpreted as
asserting that functorial Borel complexity provides a finer distinction
between the complexity of classification problems in mathematics than the
traditional notion of Borel complexity for equivalence relations

\section{Further directions and open problems}

Polish groupoids seem to be a broad generalization of Polish groups and
Polish group actions. For instance, as shown in Section \ref{Section: better}%
, any Borel action of a Polish groupoid is again a Polish groupoid. However,
no example is currently known of a Polish groupoid whose orbit equivalence
relation is not Borel bireducible with an orbit equivalence relation of a
Polish group action.

\begin{problem}
\label{Problem: groups and groupoids}Is there a Polish groupoid $G$ such
that the orbit equivalence relation $E_{G}$ is not Borel bireducible with an
orbit equivalence relation of a Polish group action? Can one find such a $G$
for which $E_{G}$ is moreover Borel?
\end{problem}

Problem \ref{Problem: groups and groupoids} has a tight connection with the
a conjecture due to Hjorth and Kechris \cite[Conjecture 1]{hjorth_new_1997}.
Recall that $E_{1}$ denotes the relation of tail equivalence for sequences
in $\left[ 0,1\right] $. Theorem 4.2 of \cite{kechris_classification_1997}
asserts that $E_{1}$ is not Borel reducible to the orbit equivalence
relation of a Polish group action. The Hjorth-Kechris conjecture asserts
that is $E$ is a Borel equivalence relation then the converse holds, namely
if $E_{1}$ is not Borel reducible to $E$ then $E$ is Borel bireducible with
the orbit equivalence relation of a Polish group action. By Corollary~\ref%
{Corollary: not reduce E1} $E_{1}$ is not reducible to any Borel orbit
equivalence relation of a Polish groupoid. Therefore a groupoid $G$ as in
Problem \ref{Problem: groups and groupoids} such that $E_{G}$ is moreover
Borel would provide a counterexample to the Hjorth-Kechris conjecture.

In the theory of Borel complexity of equivalence relations, a key role is
played by equivalence relations that are complete (or universal) for a given
class up to Borel reducibility. These equivalence relations provide natural
benchmarks of complexity. It would be interesting to analogously find
universal elements for natural classes of analytic groupoids.

\begin{problem}
Establish whether the following classes have a universal element up to Borel
reducibility: analytic groupoids, Borel groupoids, Polish groupoids,
countable groupoids.
\end{problem}

In the world of equivalence relations, the phenomenon of universality is
widespread. Most natural classes of analytic equivalence relations have
universal elements. Moreover such universal elements admit many different
descriptions. For example the following equivalence relations are all
complete for orbit equivalence relations induced by Polish group actions 
\cite%
{gao_classification_2003,melleray_geometry_2007,elliott_isomorphism_2013,farah_turbulence_2014,sabok_completeness_2013,zielinski_complexity_2014}%
:

\begin{enumerate}
\item Isomorphism of abelian C*-algebras;

\item Isomorphism of amenable, simple, unital, separable C*-algebras;

\item Isometry of separable Banach spaces;

\item Complete order isomorphism of separable operator systems;

\item Isometry of metric spaces;

\item Conjugacy of isometries of the Urysohn sphere.
\end{enumerate}

As shown in Subsections \ref{Subsection: examples} and \ref{Subsection: the
action groupoid} the relations 1-6 above are naturally the orbit equivalence
relation of a Borel groupoid. It would be interesting to known if the
functorial Borel complexity of such groupoids are different.

\begin{problem}
Consider the standard Borel groupoids associated with the relations 1-6
above. Are these groupoids Borel bi-reducible?
\end{problem}

One can consider a similar problem for the classes of countable first-order
structures that are Borel complete \cite{friedman_Borel_1989}. Recall that a
class of countable first-order structures is Borel complete if the
corresponding isomorphism relation is complete for isomorphism relations of
countable structures. Borel complete classes include countable trees,
countable linear orders, countable fields of a fixed characteristic, and
countable groups.

\begin{problem}
Are there two Borel complete classes of first order structures such that the
corresponding groupoids are not Borel bi-reducible?
\end{problem}

\appendix

\section*{\texorpdfstring{Appendix
by
Anush
Tserunyan%
\footnote{Department
of
Mathematics,
University
of
Illinois
at
Urbana-Champaign,
273
Altgeld
Hall,
1409
W.
Green
Street
(MC-382),
Urbana,
IL
61801}}{Appendix
by
Anush
Tserunyan}
\label{Appendix}}

In this appendix we show that, if $X$ is a locally Polish space, then the
Effros Borel structure on the space $F(X)$ of closed subspaces of $X$ is
standard. Recall that, as in Definition~\ref{Definition: locally Polish
space}, a topological space $X$ is \emph{locally Polish }if it has a
countable basis of open sets which are Polish in the relative topology. If $%
U $ is an open subset of $X$, we denote by $U^{-}$ the set 
\begin{equation*}
\left\{ F\in F(X):F\cap U\neq \varnothing \right\} \text{.}
\end{equation*}%
Define the Effros Borel structure on $F(X)$ to be the Borel structure
generated by the sets $U^{-}$ for $U\subset X$ open.


\begin{customthm}{A}
The Effros Borel structure on $F(X)$ is standard.
\end{customthm}

\begin{proof}
Suppose that 
\begin{equation*}
\mathcal{A}=\left\{ U_{n}:n\in \omega \right\}
\end{equation*}%
is a countable basis of Polish open subsets of $X$. For every $n\in \omega $
denote by $d_{n}$ a compatible complete metric on $U_{n}$. Clearly the
Effros Borel structure on $F(X)$ is generated by the sets $U^{-}$ for $U\in 
\mathcal{A}$. Consider the collection $\mathcal{S}_{\mathcal{A}}$%
\begin{equation*}
\left\{ U^{-},X\backslash U^{-}:U\in \mathcal{A}\right\} \text{,}
\end{equation*}%
and the topology $\tau _{\mathcal{A}}$ on $F(X)$ having $\mathcal{S}_{%
\mathcal{A}}$ as subbasis. We will show that the topology $\tau _{\mathcal{A}%
}$ on $F(X)$ is Polish. Consider the map $c$ from $F(X)$ to $2^{\omega }$
assigning to $F$ the characteristic function of $\left\{ n\in \omega :F\cap
U_{n}\neq \varnothing \right\} $. Clearly $c$ is a $\tau _{\mathcal{A}}$%
-homeomorphism onto its image. In view of~\cite[Theorem 3.11]%
{kechris_classical_1995}, in order to conclude that $\left( F(X),\tau _{%
\mathcal{A}}\right) $ is Polish it is enough to show that the image $Y$ of $%
c $ is a $G_{\delta }$ subspace of $2^{\omega }$. We claim that, for $y\in
2^{\omega }$, $y\in Y$ if and only if the following conditions hold:

\begin{enumerate}
\item for all $n,m$ with $U_n \subseteq U_m$, if $y(n) = 1$ then $y(m) = 1$;

\item for all $n$ and $\varepsilon \in \mathbb{Q}_{+}$, if $y(n)=1$ then
there is $m$ such that $y(m)=1$ and for all $i\leq n$ with $U_{i}\supseteq
U_{n}$, we have: 
\begin{equation*}
\overline{U}_{m}^{i}\subseteq U_{n}\text{ and }\mathrm{diam}%
_{i}(U_{m})<\varepsilon ,
\end{equation*}%
where the closure $\overline{U}_{m}^{i}$ and diameter $\mathrm{diam}%
_{i}(U_{m})$ are taken with respect to the metric $d_{i}$.
\end{enumerate}

Since necessity is obvious, we check that these conditions are sufficient.
Let $y\in 2^{\omega }$ satisfy conditions (i) and (ii), and define the $\tau
_{\mathcal{A}}$-closed subset of $X$ 
\begin{equation*}
F=\left\{ x\in X:\forall n\in \omega \text{, }x\in U_{n}\Rightarrow
y(n)=1\right\} \text{.}
\end{equation*}%
We show that $c(F)=y$. Fix $n\in \omega $ and note that if $y(n)=0$, then $%
F\cap U_{n}=\varnothing $ by definition. So suppose $y(n)=1$ and we have to
find an $x\in F\cap U_{n}$. Iterating (ii), we get a sequence $%
(U_{n_{i}})_{i\in \omega }$ with $n_{0}=n$ and such that for all $i\in
\omega $,

\begin{itemize}
\item $y(n_{i})=1$,

\item $\overline{U}_{n_{i+1}}^{n}\subseteq U_{n_{i}}$,

\item $\mathrm{diam}_{n}(U_{n_{i}})\leq 2^{-k}$.
\end{itemize}

Thus, since the metric $d_{n}$ on $U_{n}$ is complete, we get $\left\{
x\right\} =\bigcap_{i}\overline{U}_{n_{i}}^{n},$ for some $x\in U_{n}$. It
remains to show that $x\in F$, but this easily follows from (i).
\end{proof}

\providecommand{\bysame}{\leavevmode\hbox to3em{\hrulefill}\thinspace}
\providecommand{\MR}{\relax\ifhmode\unskip\space\fi MR }
\providecommand{\MRhref}[2]{%
  \href{http://www.ams.org/mathscinet-getitem?mr=#1}{#2}
}
\providecommand{\href}[2]{#2}

\end{document}